\documentclass[11pt, reqno]{amsart}

\usepackage{enumerate}
\usepackage{graphicx,graphics}
\usepackage{amsfonts}
\usepackage{amssymb}
\usepackage{amsthm}
\usepackage{amsmath}
\usepackage{mathrsfs,color}
\usepackage{listings}
\usepackage{orcidlink}
\input{xy}
\xyoption{all}
\usepackage{tcolorbox}

\usepackage[style=numeric,sorting=nyt,giveninits=true,maxnames=99,giveninits=true,url=false]{biblatex}
\DeclareNameAlias{sortname}{family-given}
\DeclareNameAlias{default}{family-given}

\newtheorem{theorem}{Theorem}[section]
\newtheorem{lemma}[theorem]{Lemma}

\newtheorem{remark}[theorem]{Remark}
\newtheorem{corollary}[theorem]{Corollary}

\newtheorem{proposition}[theorem]{Proposition}
\newtheorem{example}[theorem]{Example}
\newtheorem{definition}[theorem]{Definition}

\begin{document}

\title[Diagonalization of nonlinear functions]{Diagonalization of nonlinear functions in finite-dimensional spaces and the case of $\mathbb R^2$}
 

\author[R. Arnau, J.C. Cortés López, A. González Cortés,
E. A. S\'{a}nchez P\'{e}rez]{Roger Arnau \orcidlink{0000-0003-2544-8875}, Juan Carlos Cortés López \orcidlink{0000-0002-6528-2155}, Álvaro González Cortés \orcidlink{0009-0001-2328-2173} and
Enrique A. S\'{a}nchez P\'{e}rez$^*$ \orcidlink{0000-0001-8854-3154
}}

\address{ Roger Arnau,  Álvaro González Cortés, 
        Enrique \ A.\ S\'{a}nchez P\'{e}rez\\ {*}Corresponding author.
              Instituto Universitario de Matem\'atica Pura y Aplicada,
              Universitat Polit\`ecnica de Val\`encia, Camino de Vera s/n, 46022
              Valencia, Spain.}
              \address{Juan Carlos Cortés López, Instituto Universitario de Matemática Multidisciplinar, Universitat Polit\`ecnica de Val\`encia, Camino de Vera s/n, 46022
              Valencia, Spain.}
            \email{ararnnot@posgrado.upv.es, jccortes@mat.upv.es, agoncor@alumni.upv.es,
            easancpe@mat.upv.es}




\begin{abstract}
This paper introduces a new theoretical framework for the diagonalization of nonlinear functions defined in finite-dimensional real Euclidean spaces. Extending classical results from linear algebra, we provide a unified setting to determine when a nonlinear map can be represented in diagonal form via a change of basis. Due to the complexity of the equations involved, the final part of the paper focuses primarily on the two-dimensional case, for which clear characterizations can be obtained. We also illustrate how these theoretical findings can be applied to classical contexts, such as differential equations, dynamical systems, and the explicit computation of higher-order compositions and inverses of functions.
\end{abstract}

\maketitle

\section{Introduction}
Although linear operators have universal applications in all areas of pure and applied science, the linear structure is not always natural in certain settings. In such cases, it may be more convenient to use less restrictive definitions that do not rely on the linear structure of the spaces. However, some of the properties of linear operators can be preserved when we do not use linearity, and this allows us to retain some of the advantages of linear maps, for example diagonalization.

Diagonalization is a fundamental procedure in linear algebra and has numerous useful applications. Recall that, in a real finite-dimensional setting, a linear function $T\colon\mathbb R^n\to\mathbb R^n$ is said to be diagonalizable if the matrix representation of $T$ with respect to some basis is diagonal. Diagonalizing a linear function can be understood as the process of finding a coordinate system in which the action of $T$ becomes as simple as possible: each basis vector is merely scaled, and no mixing of directions occurs. This ``simplification" is a significant advantage of linear diagonalizable functions, and it is the objective of this study to extend this to nonlinear functions.

While this is a general idea, it must be said that nonlinear diagonalization is a well-known and even common area of study. The question of how to define eigenvectors and eigenvalues for general maps has been addressed from many different perspectives, both theoretical \cite{calanchi} and applied \cite{bun,cho}, and the fundamental aspects are well established. However, in most of these cases, the focus is placed on fixing eigenvalues as numerical scalars and analyzing the corresponding eigenvectors. In the present paper, by contrast, we aim to determine eigenvalues as functions produced by the change of representation of the space induced by the (linear) transformation given by a change of basis. This allows us to study global properties of the operators instead of the usual local linearizations common in dynamical systems such as the classical Hartman–Grobman or Sternberg's theorems (see, for example, \cite{katok2017modern}). In this way, we adopt a practical and concrete approach, offering explicit results that can be readily applied in a range of contexts. This will be illustrated through examples presented at the end of the paper across a variety of settings.

The foundations of the concepts studied in this work were originally presented in \cite{arn}.
There, the lattice Lipschitz inequality $|T(x) - T(y)| \leq_X K |x-y|$ for mappings $T: X \to X$ in certain ordered spaces, is shown to be equivalent to a diagonal property.
An eigendecomposition-type study is relevant in our context, similarly to the lattice Lipschitz case, although the functions that play the role of eigenvalues in our setting do not satisfy any Lipschitz condition.
When the functions $T: \mathbb R^n \to \mathbb R^n$ are of class $\mathcal C^1$, the analysis of the linearization via the Jacobian matrix $DT$ plays a central role.
As in the nonlinear eigenvalue problem studied in \cite{Guttel_2017}, the eigendecomposition depends on the evaluation point of the matrix, a fact that is emphasized in the second part of the work.
In the particular case of the plane, $\mathbb R^2$, we provide explicit algebraic criteria to determine when such diagonalization is possible.

Parallel to theoretical advancements, the past decade has seen numerous practical applications for the diagonalization of nonlinear maps, particularly through the Koopman operator \cite{KoopmanSurvey2024}.
This approach offers a linearization of a finite dimensional dynamical system by embedding it onto an infinite dimensional space.
Several models have been proposed to approximate this transformation using data-driven techniques, including kernel methods, dynamic mode decomposition, or deep neural networks \cite{vinuesa2021koopman, KernelKoopman2025, NeuralDMD2022}.
While our work shares the goal of isolating one-dimensional dynamics, we focus on the explicit criteria that determine when the diagonalization is possible through a change of basis.
 
To do this, we structure the paper as follows. After this introduction, in Section \ref{sec:diag} we present the context of the problem of diagonalization of non-linear functions. To this end, we will generalize known definitions and results from the linear case, showing illustrative examples of this new scenario. In addition, we will study in Section \ref{2dim} our proposal in more depth for the case of the 2-dimensional Euclidean space, where we will obtain complete characterizations. Finally, Section \ref{appl} is devoted to further explore the potential applications of non-linear diagonalization, such as the study of ordinary differential equations. 

We use standard notation from linear algebra and mathematical analysis \cite{ali2,hornmatrix}. Concepts and notations that might be confusing will be explained when they are introduced. 
The finite dimensional space $\mathbb R^n$ is assumed to be endowed with the Euclidean norm. A function  $\Phi: \Omega \subseteq \mathbb R^n \to \mathbb R^n$ is said to be Lipschitz if there exists a constant $K>0$ such that 
 $$
 \big\| \Phi(x_1)- \Phi(x_2) \big\| \le K \| x_1-x_2\|,
 $$
 for every $x_1, x_2 \in \Omega$. The infimum of such constants $K$ is the Lipschitz constant of $\Phi$, and it is denoted by $\|\Phi\|_{Lip}$.

\section{General framework for nonlinear diagonalization}\label{sec:diag}

We start this section by introducing the main definitions and concepts related to the diagonalization problem in the nonlinear case. We then explore how some of the well-known results in the linear case can be extended to the nonlinear setting.

Throughout this paper $\Omega$ denotes an open subset of $\mathbb R^n$, and $\Phi$ denotes a function $\Phi\colon\Omega\to\mathbb{R}^n$. We use the usual notation $\Phi(x_1,\ldots,x_n)$ to denote the image of the vector $x\in\Omega$ whose coordinates in the canonical basis $\mathcal C=\{e_1,\ldots,e_n\}$ are $(x_1,\ldots,x_n)$, that is, $$\Phi(x)=\Phi(x_1e_1+\ldots+x_ne_n)=\Phi(x_1,\ldots,x_n).$$ 
We also write the coordinate $i\in\{1,\ldots,n\}$ of the image of $\Phi$ as $\Phi_i$, so
\[
    \Phi(x)=\Phi_1(x)e_1+\ldots+\Phi_n(x)e_n=(\Phi_1(x),\ldots,\Phi_n(x)),
\]
and each component $\Phi_i$ is regarded as a real function of the $n$ variables $x_1,\ldots,x_n$. As usual, we say that $\Phi_i$ depends only on $x_i$ if $\Phi_i(x)=\Phi_i(y)$ for every $x,y\in\Omega$ with $x_i=y_i$. In that case, we simply write $\Phi_i(x_i)$ for that common value. This notation is also used for any other basis when there is no confusion.

\begin{definition}\label{def:diagonal}
    We say that a function $\Phi\colon\Omega\subseteq\mathbb{R}^n\to\mathbb{R}^n$ is \emph{diagonal}
    if $\Phi_i$ depends only on $x_i$ for every $i\in\{1,\ldots,n\}$, so
    \[
        \Phi(x_1,\ldots,x_n)=(\Phi_1(x_1),\ldots,\Phi_n(x_n))
    \]
    for every $x\in\Omega$. That is, the $i$-th component of the image of $x$ only depends on the $i$-th coordinate of $x$ in the canonical basis $\mathcal C$.
\end{definition}

\begin{remark}\label{remark:convex-domain}
    Let $\Omega\subseteq\mathbb R^n$ be open and let $\Phi\colon\Omega\to\mathbb R^n$ be $\mathcal C^1$. If $\Phi$ is diagonal, it is clear that the jacobian matrix $D\Phi(x)$ is diagonal for all $x\in\Omega$. However, the converse does not hold for a general domain $\Omega$. For instance, consider in $\mathbb R^2$ the open set $\Omega=\{(x,y)\in\mathbb R^2:y\neq 0\}$ and $\Phi\colon\Omega\to\mathbb R^2$ defined by
    \[
        \Phi(x,y)=\begin{cases}
            (x,y) & \text{if }y>0,\\
            (x+1,y) & \text{if }y<0.
        \end{cases}
    \]
    Observe that $\Phi$ is $\mathcal C^\infty$ in $\Omega$, and $D\Phi(x,y)=\operatorname{diag}(1,1)$ for all $(x,y)\in\Omega$. Nevertheless, $\Phi$ is not diagonal in the sense of Definition \ref{def:diagonal}. For instance, note that $\Phi(1,1)=(1,1)$ and $\Phi(1,-1)=(2,-1)$. That is, the first component of the image takes different values for the same $x=1$.

    The converse does hold whenever $\Omega$ is convex. Indeed, assume that
    $D\Phi(x)$ is a diagonal matrix for every $x\in\Omega$, fix $i\in\{1,\ldots,n\}$ and take $x,x'\in\Omega$ with $x_i=x_i'$. By convexity, the segment $\gamma(t)=(1-t)x+tx'$ is contained in $\Omega$ for all $t\in[0,1]$. Since $\partial\Phi_i/\partial x_j\equiv 0$ for every $j\neq i$, then
    \[
        \frac{d}{dt}\,\Phi_i(\gamma(t))
        =\sum_{j=1}^{n}\frac{\partial\Phi_i}{\partial x_j}(\gamma(t))\,(x_j'-x_j)
        =\frac{\partial\Phi_i}{\partial x_i}(\gamma(t))\,(x_i'-x_i)=0 .
    \]
    Hence $\Phi_i(x)=\Phi_i(x')$, so $\Phi_i$ only depends on the $i$-th coordinate. Therefore, $\Phi$ is diagonal. 


    In what follows, whenever we require $\Phi$ to be $\mathcal C^1$, we will assume that $\Omega$ is an open and convex subset of $\mathbb R^n$ in order to obtain this characterization.
\end{remark}

As in the linear case, a function $\Phi$ is diagonalizable if there exist a basis $\mathcal B$ of $\mathbb R^n$ in which its expression is diagonal. In this case, the linear isomorphism given by the change of coordinates from the basis $\mathcal B$ to the canonical basis $\mathcal C$ is a key element.
\begin{definition}\label{def:diagbasis}
    A function $\Phi\colon\Omega\subseteq\mathbb{R}^n\to\mathbb{R}^n$ is diagonalizable if there exists a basis $\mathcal B=\{u_1,\ldots,u_n\}$ of $\mathbb R^n$ such that the expression of $\Phi$ in that basis is diagonal. That is, there exist scalar functions $\lambda_1,\ldots,\lambda_n$ such that for any $x=a_1u_1+\ldots+a_nu_n\in\Omega$, we can write $\Phi(x)=\lambda_1(a_1)u_1+\ldots+\lambda_n(a_n)u_n$.
\end{definition}

We also propose the following definition, which is more practical and will be the one with which we will mainly work in this paper. Although equivalent to the previous definition, as we demonstrate in the following, this definition is consistent with the concept of conjugacy employed in dynamical systems.

\begin{definition}\label{def:diagiso}
    For $\Phi\colon\Omega\subseteq\mathbb{R}^n\to\mathbb{R}^n$, we say that $\Phi$ is diagonalizable if there exist a linear isomorphism $L$ and a diagonal function $\Psi$ such that $\Phi=L\circ\Psi\circ L^{-1}$. If there is no confusion, we will write the invertible matrix $P$ which represents the linear isomorphism instead of $L$, that is, $\Phi=P\circ\Psi\circ P^{-1}$.
\end{definition}

\begin{proposition}
    Definitions \ref{def:diagbasis} and \ref{def:diagiso} are consistent.
    \begin{proof}
        Consider any function $\Phi$ diagonalizable in the sense of Definition \ref{def:diagbasis}. Denote by $\mathcal B=\{u_1,\ldots,u_n\}$ the basis in which $\Phi$ is diagonal in this sense, and $\lambda_1,\ldots,\lambda_n$ the corresponding scalar functions. Let $P$ be the invertible matrix which gives the change of coordinates from $\mathcal B$ to the canonical basis. That is, if we have representations for $x\in\Omega$ as $x=x_1e_1+\ldots+x_ne_n$ and $x=a_1u_1+\ldots+a_nu_n$, then $$\begin{pmatrix}
            x_1\\\vdots\\x_n
        \end{pmatrix}=P\begin{pmatrix}
            a_1\\\vdots\\a_n
        \end{pmatrix}.$$
        In particular, for each $1 \leq i \leq n$ holds $Pe_i=u_i$. Now, note that
        \begin{align*}
            (P^{-1}\circ\Phi\circ P)\left(\sum_{i=1}^nx_ie_i\right)&=(P^{-1}\circ \Phi)\left(\sum_{i=1}^nx_iu_i\right)\\&=P^{-1}\sum_{i=1}^n\lambda_i(x_i)u_i\\&=\sum_{i=1}^n\lambda_i(x_i)e_i.
        \end{align*}
        Therefore, naming $\Psi:=P^{-1}\circ\Phi\circ P$ we have that $\Psi$ is diagonal, and thus $\Phi$ is diagonalizable in the sense of Definition \ref{def:diagiso}.

        Conversely, let $\Phi$ be diagonalizable according to Definition \ref{def:diagiso}, and denote by $P$ the matrix that diagonalizes $\Phi$ and $\Psi$ its diagonal function associated. Consider the basis $\mathcal B=\{u_1,\ldots,u_n\}$ where $u_i=Pe_i$. Then, by a similar argument to the one presented above, we can show that $\Phi(x)=\lambda_1(a_1)u_1+\ldots+\lambda_n(a_n)u_n$, being $x=a_1u_1+\ldots+a_nu_n$ and $\lambda_i:=\Psi_i$.
    \end{proof}
\end{proposition}

As in the linear case, we define the analogous concepts of eigenvector and eigenfunction in the following way.
\begin{definition}\label{def:eigenvector}
    Let $\Phi\colon\Omega\subseteq\mathbb R^{n}\to\mathbb R^{n}$. We say that a nonzero vector $u\in\mathbb R^{n}$ is an eigenvector of $\Phi$ if the set $D_{u}:=\{h\in\mathbb R : hu\in\Omega\}$ is non-empty and there exists a scalar function $\lambda\colon D_{u}\to\mathbb R$ such that
    \[
        \Phi(hu)=\lambda(h)u, \quad \text{for all } h\in D_{u}.
    \]
    That is, $\Phi$ is invariant under the direction of $u$, and $\lambda$ modulates the scaling in such direction. We call $\lambda$ an eigenfunction of $\Phi$. Note that $\lambda$ is uniquely determined by $u$, since $u\neq0$ gives $\lambda(h)=\langle\Phi(hu),u\rangle/\|u\|^{2}$ for every $h\in D_{u}$.
\end{definition}

\begin{proposition}[Composition results]
Consider $\Phi\colon\Omega\subseteq\mathbb R^n\to\Omega$ and let $u$ be an eigenvector of $\Phi$ with associated eigenfunction $\lambda$.
\begin{itemize}
    \item[(i)] For each $k\in\mathbb N$, $u$ is an eigenvector of the $k$-fold composition $\Phi^{(k)}$ and its eigenfunction is $\lambda^{(k)}$.
    \item[(ii)] Let $p$ be a polynomial $p(t)=\sum_{j=0}^N\alpha_jt^j$ and define $p(\Phi)(\cdot):= \sum_{j=0}^N\alpha_j\Phi^{(j)}(\cdot)$. Then, $u$ is an eigenvector of $p(\Phi)$ with eigenfunction $p(\lambda)$.
\end{itemize}
    \begin{proof}  
 To prove (i) note that $$\Phi^{(2)}(hu)=\Phi(\Phi(hu))=\Phi(\lambda(h)u)=\lambda(\lambda(h))u=\lambda^{(2)}(h)u.$$
    The desired result can be straightforwardly obtained from here by recursion. Finally, the result of (ii) is a consequence of (i) since
    $$p(\Phi)(hu)=\sum_{j=0}^N\alpha_j\Phi^{(j)}(hu)=\sum_{j=0}^N\alpha_j\lambda^{(j)}(h)u=p(\lambda)(h)u.$$
    \end{proof}
\end{proposition}

For a linear map, it is well known that eigenvectors associated to different eigenvalues are linearly independent. However, this situation may not hold for the nonlinear case. For example, consider $\Phi(x,y)=(x^2+y^2,2xy)$ and note that
\begin{align*}
    \Phi(h(1,0))&=(h^2,0)=h^2(1,0),\\
    \Phi(h(2,0))&=(4h^2,0)= 2h^2(2,0).
\end{align*}
That is, $(1,0)$ and $(2,0)$ are linearly dependent eigenvectors but their respective eigenfunctions are different. However, we can achieve a result similar to that in the linear case by imposing a condition that avoids the previously shown situation.

\begin{proposition}\label{prop:independence}
    Let $\Phi\colon\Omega\subseteq\mathbb R^{n}\to\mathbb R^{n}$ and let $u,v\in\mathbb R^{n}$ be eigenvectors of $\Phi$ with eigenfunctions $\lambda$ and $\mu$, respectively. If $u=\alpha v$ for some $\alpha\in\mathbb R\setminus\{0\}$, then $D_{u}=\alpha^{-1}D_{v}$ and $\alpha\lambda(h)=\mu(\alpha h)$ for every $h\in D_{u}$.
    
    Consequently, if for every $\alpha\in\mathbb R\setminus\{0\}$ such that $D_{u}\cap\alpha^{-1}D_{v}\neq\emptyset$ there exists $h\in D_{u}\cap\alpha^{-1}D_{v}$ with $\alpha\lambda(h)\neq\mu(\alpha h)$, then $u$ and $v$ are linearly independent.
\end{proposition}

\begin{proof}
    Assume $u=\alpha v$ with $\alpha\neq0$. Then
    \[
        D_{u}=\{h\in\mathbb R : h\alpha v\in\Omega\}=\{h\in\mathbb R : \alpha h\in D_{v}\} =\alpha^{-1}D_{v}.
    \] 
    Let $h\in D_{u}$. On the one hand,
    \begin{equation}\label{eq:phiv}
        \Phi(hu)=\lambda(h)u=\alpha\lambda(h)v,
    \end{equation}
    and on the other hand,
    \begin{equation}\label{eq:phiu}
        \Phi(hu)=\Phi(h\alpha v)=\mu(h\alpha)v.
    \end{equation}
     Since $v\neq 0$, from \eqref{eq:phiv} and \eqref{eq:phiu} it follows that $\alpha\lambda(h)=\mu(h\alpha)$. Since $h\in D_{u}$ is arbitrary we conclude the proof.
\end{proof}

\begin{proposition}
Suppose that $\Phi\colon\mathbb R^n\to\mathbb R^n$ is a diagonalizable function. Then, its eigenfunctions are of the form $\lambda_j(h)=c_jh$, where $c_j\in\mathbb R$, for $j=1,\ldots,n$ if and only if $\Phi$ is linear.
    \begin{proof}
    Since $\Phi$ is diagonalizable, let $u_1,\ldots,u_n$ be the eigenvectors of $\Phi$ which form that basis in which $\Phi$ is diagonal. Consider $\lambda_j$ as the eigenfunction of $u_j$, and consider $a_1,\ldots,a_n\in\mathbb R$. Therefore,
    $$\Phi\left(\sum_{j=1}^n a_ju_j\right)=\sum_{j=1}^n\lambda_j(a_j)u_j=\sum_{j=1}^nc_ja_ju_j=\sum_{j=1}^na_j\lambda_j(1)u_j=\sum_{j=1}^na_j\Phi(u_j),$$
    and consequently $\Phi$ is linear.
    To get the converse result, note that if $\Phi$ is linear and diagonalizable in the sense of our definitions, then its eigenvalues are real numbers $c_1,\ldots,c_n$. So, for its eigenvectors $u_1,\ldots,u_n$ we get $$\Phi(hu_j)=h\Phi(u_j)=hc_ju_j=\lambda_j(h)u_j,$$ for $\lambda_j(h)=c_jh$.
    \end{proof}
\end{proposition}

\begin{proposition}[Inversion results]
For a function $\Phi\colon\Omega\subseteq\mathbb R^n\to\mathbb R^n,$ the following statements hold:
    \begin{enumerate}
        \item[(i)] If there exists an eigenvector $u\in\mathbb R^n$ whose eigenfunction $\lambda$ is not injective on $D_u$, then $\Phi$ is not injective either.
        \item[(ii)] Suppose that $\Phi$ is invertible on its image and denote its inverse by $\Phi^{-1}\colon\Phi(\Omega)\to\Omega$. If $u\in\mathbb R^n$ is an eigenvector for the eigenfunction $\lambda$, then $\lambda\colon D_u\to\lambda(D_u)$ is invertible.
        \item[(iii)] Under the assumptions of (ii), it holds that
        $\Phi^{-1}(tu)=\lambda^{-1}(t)\,u$ for every $t\in\lambda(D_u)$. If moreover $\Phi(\Omega)\cap\mathbb R u=\Phi(\Omega\cap\mathbb R u)$, then $u$ is an eigenvector of $\Phi^{-1}$ with eigenfunction $\lambda^{-1}$.
    \end{enumerate}
\end{proposition}
\begin{proof}
    To get (i), let $h_1,h_2\in D_u$ be such that $h_1\neq h_2$ and $\lambda(h_1)=\lambda(h_2)$.
    Since $u\neq0$, the points $h_1u$ and $h_2u$ of $\Omega$ are different, and
    \[
        \Phi(h_1u)=\lambda(h_1)u=\lambda(h_2)u=\Phi(h_2u),
    \]
    so $\Phi$ is not injective. 
    
    The result of (ii) is an immediate consequence of (i), since the injectivity of $\Phi$ forces $\lambda$ to be injective on $D_u$, and we have restricted its codomain to its range.

    Finally, let $t\in\lambda(D_u)$ and take $h\in D_u$ with $\lambda(h)=t$, so that $h=\lambda^{-1}(t)$ by (ii). Then $tu=\lambda(h)u=\Phi(hu)$, and therefore
    \begin{equation}\label{eq:eigen-phi-inverse}
        \Phi^{-1}(tu)=\Phi^{-1}\bigl(\Phi(hu)\bigr)=hu=\lambda^{-1}(t)\,u .
    \end{equation}

    Moreover, note that $\lambda(D_u)\subseteq\{t\in\mathbb R : tu\in\Phi(\Omega)\}$. Suppose now that $\Phi(\Omega)\cap\mathbb R u=\Phi(\Omega\cap\mathbb R u)$, and let $t\in\mathbb R$ be such that $tu\in\Phi(\Omega)$. Then $tu\in\Phi(\Omega\cap\mathbb R u)$, so $tu=\Phi(hu)=\lambda(h)u$ for some $h\in D_u$. Since $u\neq0$, then $t=\lambda(h)\in\lambda(D_u)$, hence
    \[
        \lambda(D_u)=\{t\in\mathbb R : tu\in\Phi(\Omega)\}
    \]
    Since both sets coincide, the identity \eqref{eq:eigen-phi-inverse} holds for every $t$ with $tu\in\Phi(\Omega)$, thus $u$ is an eigenvector of $\Phi^{-1}$ with eigenfunction $\lambda^{-1}$.
\end{proof}

The converse of these inversion results does not necessarily hold in general. Consider the function $\Phi(x,y)=(e^x,0)$. Since $\Phi(h(1,0))=(e^h,0)=\lambda(h)(1,0)$, then $(1,0)$ is an eigenvector of $\Phi$ for the injective eigenfunction $\lambda(h)=e^h$. However, $\Phi$ is not injective since $\Phi(a,b)=\Phi(a,c)$ for any $a$ and $b\neq c$. And more obviously, $\Phi$ may be invertible and have no eigenvectors, such as $\Phi(x,y)=(-y,x)$ whose inverse is $\Phi^{-1}(x,y)=(y,-x)$.

In general, the relationship between the diagonalization problem and the eigenvectors and eigenfunctions is weaker than in the linear case. In the linear case, if there exists a set of $n$ linearly independent eigenvectors of a function $T$, then it is known that $T$ is diagonalizable in the basis defined by these vectors. However, this is not the case for nonlinear functions $\Phi$, as demonstrated in the following example.
\begin{example}
    Consider $\Phi\colon\mathbb{R}^2\to\mathbb{R}^2$ defined by $\Phi(x,y)=(y^2,xy)$. Note that
    \begin{align*}
        \Phi(h(1,1))&=(h^2,h^2)=h^2(1,1),\\
        \Phi(h(1,-1))&=(h^2,-h^2)=h^2(1,-1).
    \end{align*}
    Now, for the matrix $P=\begin{pmatrix}
        1 & 1\\
        1 & -1
    \end{pmatrix}$ which inverse is $P^{-1}=\frac{1}{2}\begin{pmatrix}
        1 & 1\\
        1 & -1
    \end{pmatrix}$, we observe that $\Psi=P^{-1}\circ\Phi\circ P$ is not diagonal. Indeed, for $(s,t)\in\mathbb{R}^2$, we get
    \begin{align*}
        \Psi(s,t)&=(P^{-1}\circ\Phi\circ P)(s,t)\\&=(P^{-1}\circ\Phi)(s+t,s-t)\\&=P^{-1}\Phi(s+t,s-t)\\&=P^{-1}\begin{pmatrix}
        s^2-2st+t^2\\s^2-t^2
    \end{pmatrix}=\begin{pmatrix}
        s^2-st\\ t^2-st
    \end{pmatrix}.
    \end{align*}
\end{example}

Nevertheless, it is possible to get a converse of this result, as we will show below. The key that relates the diagonality with the eigendecomposition is that the function must vanish at $0\in\mathbb R^n$ (origin preserving). Note that this requirement can be satisfied by any function after a translation $\Psi(x) = \Phi(x) - \Phi(0)$.
This condition can be interpreted in different settings. For example, in dynamical systems it means that 0 is a fixed point of $\Phi$. On the other hand, there are certain classes of mappings of high interest that are origin preserving, such as homogeneous functions with positive degree or, in particular, all linear functions.

\begin{proposition}\label{prop:eigenvalues}
    Let $\Omega\subseteq\mathbb R^n$ be open and such that $0\in\Omega$. Let $\Phi\colon\Omega\to\mathbb R^n$ be a function such that $\Phi(0)=0$. If $\Phi$ is diagonalizable, then there exist $n$ linearly independent eigenvectors $u_1,\ldots,u_n$ with eigenfunctions $\lambda_i\colon D_{u_i}\to\mathbb R$, not necessarily different.
\begin{proof}
    Since $\Phi$ is diagonalizable, there exists a basis $\mathcal B=\{u_1,\ldots,u_n\}$ of $\mathbb{R}^n$ such that $\Phi(a_1u_1+\ldots+a_nu_n)=\lambda_1(a_1)u_1+\ldots+\lambda_n(a_n)u_n$. Note that the condition $\Phi(0)=0$ leads to $\lambda_i(0)=0$ for $i=1,\ldots,n$ since $\mathcal B$ is a basis of $\mathbb R^n$. Therefore, fixing $1\leq j\leq n$ and considering $a_i=0$ if $i\neq j$, then
    $$\Phi(a_ju_j)=\lambda_1(0)u_1+\ldots+\lambda_j(a_j)u_j+\ldots+\lambda_n(0)u_n=\lambda_j(a_j)u_j.$$
\end{proof}
\end{proposition}
In the following example, the necessity of the condition $\Phi(0)=0$ will be shown.
\begin{example}
    Consider $\Phi\colon\mathbb{R}^2\to\mathbb{R}^2$ defined as $\Phi(x,y)=(e^y+x,e^y+y)$, where $\Phi(0,0)=(1,1)$. This function is diagonalizable, for example, in the basis $\{(1,1),(1,0)\}$. Indeed, if $P=\begin{pmatrix}
        1 &1\\ 1&0
    \end{pmatrix}$ and so $P^{-1}=\begin{pmatrix}
        0 & 1\\1&-1
    \end{pmatrix},$ then
    $$(P^{-1}\circ\Phi\circ P)(s,t)=P^{-1}\circ \Phi(s+t,s)=P^{-1}\begin{pmatrix}
        e^s+s+t\\ e^s+s
    \end{pmatrix}=\begin{pmatrix}
    e^s+s\\t
    \end{pmatrix}.$$
The method for finding this basis will be explained later and is provided by Theorem \ref{teo:r2char}, but this is not the objective of the current example. Besides, although $(1,1)$ is an eigenvector of $\Phi$ since
    $$\Phi(h(1,1))=(e^h+h,e^h+h)=\lambda_1(h)(1,1), \quad \text{for}  \ \lambda_1(h)=e^h+h,$$
the vector $(1,0)$ is not an eigenvector, because $\Phi(h(1,0))=(1+h,1)$. Nevertheless, although $\Phi$ is not invariant under the direction of $(1,0)$, it is invariant under a parallel direction: $\Phi(h,0)=\Phi(0)+\lambda_2(h)(1,0)$, for $\lambda_2(h)=h$.
\end{example}

To establish the main characterization of diagonalizable functions, let us recall a classic result from linear algebra, see for example \cite{hornmatrix}. 
\begin{lemma}[Theorem 1.3.21 of \cite{hornmatrix}]\label{lemma:conmuting}
    Consider an arbitrary family of diagonalizable matrices $\{A_{\alpha}\}_{\alpha\in\Delta}$. Then, it is a commuting family if and only if it is a simultaneously diagonalizable family, that is, there exists a non singular $P$ such that $PA_{\alpha}P^{-1}$ is diagonal for all $\alpha\in\Delta$.
\end{lemma}
Therefore, the characterization to the nonlinear case is as follows.

\begin{theorem}\label{theo:simcar}
    Let $\Omega\subseteq\mathbb R^n$ be open and convex, and let $\Phi\colon\Omega\to\mathbb{R}^n$ be $\mathcal{C}^1$. Then, $\Phi$ is diagonalizable if and only if $\{D\Phi(x)\}_{x\in\Omega}$ is a commuting family of diagonalizable Jacobian matrices.
\end{theorem}

\begin{proof}
    Assume first that $\Phi$ is diagonalizable. By Definition \ref{def:diagiso} there exists an invertible matrix $P$ such that $\Psi:=P^{-1}\circ\Phi\circ P$ is diagonal on $P^{-1}\Omega$. By Remark \ref{remark:convex-domain} the matrix $D\Psi(s)$ is diagonal for every $s\in P^{-1}\Omega$. Furthermore, by the chain rule we get $D\Psi(s)=P^{-1}D\Phi(Ps)P$, so $P^{-1}D\Phi(x)P$ is diagonal for every $x\in\Omega$. Thus $\{D\Phi(x)\}_{x\in\Omega}$ is a simultaneously diagonalizable family and, by Lemma \ref{lemma:conmuting}, it is a commuting family of diagonalizable matrices.

    Conversely, assume that $\{D\Phi(x)\}_{x\in\Omega}$ is a commuting family of diagonalizable matrices. By Lemma \ref{lemma:conmuting} there exists a nonsingular matrix $R$ such that $R\,D\Phi(x)\,R^{-1}$ is diagonal for every $x\in\Omega$. Set $P:=R^{-1}$ and $\Psi:=P^{-1}\circ\Phi\circ P$. Note that $\Psi$ is defined on $P^{-1}\Omega=R\,\Omega$, which is open and convex since $R$ is a linear isomorphism. By the chain rule,
    \[
        D\Psi(s)=P^{-1}D\Phi(Ps)P=R\,D\Phi(Ps)\,R^{-1}
    \]
    is diagonal for every $s\in P^{-1}\Omega$, so Remark \ref{remark:convex-domain} gives that $\Psi$ is diagonal. Therefore $\Phi=P\circ\Psi\circ P^{-1}$ is diagonalizable.
\end{proof}

In the linear case it is known that the orthogonal diagonalization characterizes the symmetric matrices. The next results explore whether this holds for the nonlinear case.
\begin{proposition}
    Let $\Phi\colon\Omega\subseteq\mathbb{R}^n\to\mathbb{R}^n$ be a $\mathcal{C}^1$ function. Suppose there exists an orthogonal matrix $Q$ that diagonalizes $\Phi$. Then, the Jacobian matrix $D\Phi(x)$ is symmetric for all $x\in\Omega$.
    \begin{proof}
        By hypothesis, the function $\Psi:=Q^t\circ\Phi\circ Q$ is diagonal, hence its Jacobian matrix $D\Psi(s)$ is diagonal for each $s\in Q^t\Omega$ (see Remark \ref{remark:convex-domain}). In particular, it is symmetric. Moreover, by the chain rule we get $D\Psi(s)=Q^tD\Phi(x)Q$, being $x=Qs$, and therefore $D\Phi(x)= QD\Psi(s)Q^t$. So, we conclude that
        $$D\Phi(x)^t=Q D\Psi(s)^t Q^t=QD\Psi(s)Q^t= D\Phi(x),$$
        that is, $D\Phi(x)$ is symmetric for each $x\in\Omega$.
    \end{proof}
\end{proposition}
The converse of this result does not hold in general. There are functions with symmetric Jacobian matrices that are not diagonalizable, so the well-known characterization in the linear case can not be extended to our non-linear setting. Consider, for example, $\Phi\colon\mathbb{R}^2\to\mathbb{R}^2$ defined as $\Phi(x,y)=(x+y,y^2+x)$. Its Jacobian matrix is
$$D\Phi(x,y)=\begin{pmatrix}
    1 & 1\\ 1 & 2y
\end{pmatrix},$$
which is clearly symmetric. However, this family of Jacobian matrices does not commute, and thus $\Phi$ is not diagonalizable by Theorem \ref{theo:simcar}.
The problem here is that, while $D\Phi(x,y)$ is diagonalizable at any $(x,y)$, the eigenvectors differ between points.
Nevertheless, if we assume that $\Phi$ is diagonalizable and its Jacobian is symmetric, then the result holds.

\begin{proposition}\label{prop:orthogonal}
    Let $\Omega\subseteq\mathbb R^n$ be open and convex and let $\Phi\colon\Omega\to\mathbb R^n$ be a $\mathcal C^1$ function. If $\Phi$ is diagonalizable and $D\Phi(x)$ is symmetric for every $x\in\Omega$, then there exists an orthogonal matrix $Q$ such that $\Psi:=Q^{t}\circ\Phi\circ Q$ is diagonal.
\end{proposition}
\begin{proof}
    Since $\Phi$ is diagonalizable, by Theorem \ref{theo:simcar} it follows that
    $\{D\Phi(x)\}_{x\in\Omega}$ is a commuting family. Therefore, by Lemma \ref{lemma:conmuting} there exists a nonsingular matrix $P$ such that diagonalizes $D\Phi(x)$ for all $x\in\Omega$. In particular, since all these matrices $D\Phi(x)$ are symmetric, there exists an orthogonal matrix $Q$ such that $Q^{t}D\Phi(x)Q$ is diagonal for every $x\in\Omega$.

    Consider $\Psi:=Q^{t}\circ\Phi\circ Q$, defined on $Q^{t}\Omega$, which is open and convex because $Q^{t}$ is a linear isomorphism. By the chain rule, $D\Psi(s)=Q^{t}D\Phi(Qs)Q$ for every $s\in Q^{t}\Omega$, which is a diagonal matrix. From Remark \ref{remark:convex-domain} it follows that $\Psi$ is diagonal.
\end{proof}

\section{Complete characterization of the $2-$dimensional case} \label{2dim}

The aim of this section is to provide an alternative and more practical characterization of diagonalizable functions, which also yields an explicit basis with respect to which the function is diagonal. Due to the technical complexity of this result, for the sake of clarity we restrict the results of this section to functions on $\mathbb R^2$.

If $\mathcal B$ is a basis of $\mathbb{R}^2$, we denote by $P_{\mathcal B, \mathcal C}$ the change of basis matrix from the coordinates in $\mathcal B$ to the coordinates in the canonical basis $\mathcal C$, and $P_{\mathcal C, \mathcal B}$ for the inverse change. Recall that $P_{\mathcal C, \mathcal B} = P_{\mathcal B, \mathcal C}^{-1}$. In particular, writing $\mathcal B = \{ v_1, v_2 \} = \big\{ (\alpha_{11}, \alpha_{12}), (\alpha_{21}, \alpha_{22}) \big\}$, we have the change of basis matrix
\begin{equation*}
	P_{\mathcal B, \mathcal C} =
	\begin{pmatrix} \alpha_{11} & \alpha_{21} \\ \alpha_{12} & \alpha_{22} \end{pmatrix}
	\quad \text{and} \quad
	P_{\mathcal C, \mathcal B} = \frac{1}{\det(P_{\mathcal B, \mathcal C})}
	\begin{pmatrix} \alpha_{22} & - \alpha_{21} \\ - \alpha_{12} & \alpha_{11} \end{pmatrix},
\end{equation*}
where $\det(P_{\mathcal B, \mathcal C}) = \alpha_{11} \alpha_{22} - \alpha_{12} \alpha_{21}$.

Consider now the mapping $\Psi = (\Psi_1, \Psi_2) : \Tilde{\Omega} \subseteq \mathbb R^2 \to \mathbb R^2$ to be $\Phi$ in terms of the coordinates of the new basis $\mathcal B$, that is,
\begin{equation*}
	\Psi(s,t) = \Phi( s v_1 + t v_2 ) = \Psi_1(s,t) v_1 + \Psi_2(s,t) v_2.
\end{equation*}
The relation between the definition sets is $\Tilde{\Omega} = P_{\mathcal C, \mathcal B}\Omega = \big\{ (s,t) \in \mathbb R^2 : s v_1 + t v_2 \in \Omega \big\}$.
The relation between the mapping in its different coordinates is given by $\Psi=P_{\mathcal B, \mathcal C}^{-1}\circ\Phi\circ P_{\mathcal B, \mathcal C}$, and so
\begin{equation}
	\label{eq:latlip_Psi_matrix}
	\begin{pmatrix} \Psi_1 \\ \Psi_2 \end{pmatrix} =
	\frac{1}{\det(P_{\mathcal B, \mathcal C})} \begin{pmatrix} \alpha_{22} & - \alpha_{21} \\ - \alpha_{12} & \alpha_{11} \end{pmatrix} \circ
	\begin{pmatrix} \Phi_1 \\ \Phi_2 \end{pmatrix} \circ
	\begin{pmatrix} \alpha_{11} & \alpha_{21} \\ \alpha_{12} & \alpha_{22} \end{pmatrix}.
\end{equation}
Expanding in terms of $\alpha_{ij}$, the previous composition becomes
\begin{equation*}
	\left\{ 
	\begin{aligned}
		\det(P_{\mathcal B,\mathcal C}) \Psi_1(s,t) = & \alpha_{22} \Phi_1( \alpha_{11} s + \alpha_{21} t, \alpha_{12} s + \alpha_{22} t ) \\
		&   - \alpha_{21} \Phi_2( \alpha_{11} s + \alpha_{21} t, \alpha_{12} s + \alpha_{22} t )
		\quad \text{and} \\
		\det(P_{\mathcal B,\mathcal C}) \Psi_2(s,t) = & - \alpha_{12} \Phi_1( \alpha_{11} s + \alpha_{21} t, \alpha_{12} s + \alpha_{22} t ) \\
		&  + \alpha_{11} \Phi_2( \alpha_{11} s + \alpha_{21} t, \alpha_{12} s + \alpha_{22} t ),
	\end{aligned}
	\right.
\end{equation*}
for any $(s,t) \in \Tilde{\Omega}$ or, equivalently, $( \alpha_{11} s + \alpha_{21} t, \alpha_{12} s + \alpha_{22} t ) \in \Omega$.

To establish the main characterization of the diagonalizable functions, we will first present a key lemma  on the existence of a basis in which $\Phi$ is diagonal. The mappings are assumed to be $\mathcal C^1$ and defined on open and convex subsets $\Omega \subseteq \mathbb R^2$. In particular, they can be defined on the whole $\mathbb R^2$.
We denote by $\varphi_{ij}$ the partial derivative of $\Phi_i$ over its $j$-component, that is, the Jacobian of $\Phi$ is given by
\begin{equation}
	D\Phi (x,y) =
	\begin{pmatrix}
		\nabla \Phi_1(x,y) \\
		\nabla \Phi_2(x,y)
	\end{pmatrix} =
	\begin{pmatrix}
		\varphi_{11}(x,y) & \varphi_{12}(x,y) \\
		\varphi_{21}(x,y) & \varphi_{22}(x,y)
	\end{pmatrix}.
\end{equation}

\begin{lemma}
	\label{lemma:latlip_quadratic_eq}
	Let $\Phi: \Omega \subseteq \mathbb R^2 \to \mathbb R^2$ be a $\mathcal{C}^1$ function. Consider the quadratic form equation
	\begin{equation}
		\label{eq:latlip_quadratic_b}
		0 = - \beta_1^2 \varphi_{21}(x,y) + \beta_1 \beta_2 \big( \varphi_{11}(x,y) - \varphi_{22}(x,y) \big)
		+ \beta_2^2 \varphi_{12}(x,y)
	\end{equation}
    Then, there exists a basis of $\mathbb R^2$ such that $\Phi$ is diagonal if and only if there exist two linearly independent vectors $( \hat \beta_1, \hat \beta_2 )$ and $( \bar \beta_1, \bar \beta_2 )$, not depending on $(x,y)$, which solve \eqref{eq:latlip_quadratic_b} for all $(x,y)\in\Omega$ \footnote{Note that (0,0) always solves \eqref{eq:latlip_quadratic_b}, so the content of the statement is the existence of two non-trivial solutions linearly independent.}. If this holds, $\Phi$ is diagonal with respect to the basis given by the solutions.
\end{lemma}

\begin{proof}
	Following the notation introduced before, the mapping $\Phi$ is diagonal in the basis $\mathcal B = \{ v_1, v_2 \} = \big\{ (\alpha_{11}, \alpha_{12}), (\alpha_{21}, \alpha_{22}) \big\}$ if and only if $\Psi_1(s,t)$ only depends on $s$ and $\Psi_2(s,t)$ on $t$.
	Equivalently, by Remark \ref{remark:convex-domain} we get
	\begin{equation}
		\label{eq:latlip_lemma_partials_0_proof}
		\dfrac{ \partial \Psi_1(s,t) }{ \partial t } = 0 \quad \text{and} \quad \dfrac{ \partial \Psi_2(s,t) }{ \partial s } = 0, \quad \forall (s,t) \in \Tilde{\Omega},
	\end{equation}
	which can be written in terms of $\Phi$ as
	\begin{equation*}
		\left\{
		\begin{aligned}
			0 & = \alpha_{22}^2 \varphi_{12}(x,y) + \alpha_{21} \alpha_{22} \big( \varphi_{11}(x,y) - \varphi_{22}(x,y) \big) - \alpha_{21}^2 \varphi_{21}(x,y), \\
			0 & = - \alpha_{12}^2 \varphi_{12}(x,y) + \alpha_{11} \alpha_{12} \big( \varphi_{22}(x,y) - \varphi_{11}(x,y) \big) + \alpha_{11}^2 \varphi_{21}(x,y),
		\end{aligned}
		\right.
	\end{equation*}
	for all $(x,y) = ( \alpha_{11} s + \alpha_{21} t, \alpha_{12} s + \alpha_{22} t ) \in \Omega$. Therefore, from \eqref{eq:latlip_lemma_partials_0_proof} it follows that both vectors of the basis $\mathcal B$ are solutions of the same quadratic equation \eqref{eq:latlip_quadratic_b}.
    
    Therefore, if $\Phi$ is diagonal in the basis $\mathcal B$, then $( \alpha_{11}, \alpha_{12} )$ and $( \alpha_{21}, \alpha_{22} )$ are two solutions of \eqref{eq:latlip_quadratic_b}.
	As \eqref{eq:latlip_lemma_partials_0_proof} is satisfied for any $(s,t) \in \Tilde{\Omega}$, then \eqref{eq:latlip_quadratic_b} holds for any $(x,y) \in \Omega$. The linear independence of the solutions is consequence of the fact that $\mathcal B$ is a basis.
	
	Conversely, if two linearly independent solutions $( \hat \beta_1, \hat \beta_2 )$ and $( \bar \beta_1, \bar \beta_2 )$ exist, one can consider the basis $\mathcal B =  \big\{ ( \hat \beta_1, \hat \beta_2 ), ( \bar \beta_1, \bar \beta_2 ) \big\}$, which clearly satisfies \eqref{eq:latlip_lemma_partials_0_proof}, and so $\Phi$ has a diagonal representation in $\mathcal B$. Then, the corresponding $\Psi$ is diagonal.
\end{proof}

In other words, previous lemma states that there must exist two directions in the plane that are invariant under the Jacobian matrix $D\Phi(x,y)$ at every point.

Note that if $\Phi$ is diagonal with respect to a basis $\mathcal B = \{ v_1, v_2 \}$, any other basis formed by multiples of it such as $\mathcal D = \{ \lambda v_1, \mu v_2 \}$ with $\lambda, \mu \neq 0$ also satisfies the condition.
Denoting by $\hat \Psi : \Tilde \Omega \subseteq \mathbb R^2 \to \mathbb R^2$ the functions \eqref{eq:latlip_Psi_matrix} with respect to $\mathcal D$, we have that 
\begin{equation*}
	\hat \Psi_1 (s) = \lambda^{-1}\Psi_1 \left( \lambda s \right), \quad \text{and} \quad
	\hat \Psi_2 (t) = \mu^{-1} \Psi_2 \left(\mu t \right), \quad \forall (s,t) \in \Tilde\Omega.
\end{equation*}

In what follows, we will omit the variables $(x,y)$ when referring to $\varphi_{i,j}$, since the dependence is clear.

Before introducing the main result let us present the following Lemma.

\begin{lemma}\label{lemma:algebra}
    Let $w\colon \mathbb{R}^{2}\to \mathbb{R}^{3}$ be the mapping defined by $w(s,t)=(t^{2},st,-s^{2})$. Assume that $u=(u_1,u_2)$ and $v=(v_1,v_2)$ are linearly independent vectors in $\mathbb{R}^2$. Then the images $w(u)$ and $w(v)$ are linearly independent in $\mathbb{R}^3$. Furthermore, for any non-zero vector $(a,b,c)\in\mathbb{R}^3$ orthogonal to both $w(u)$ and $w(v)$ it holds $b^2+4ac>0$.
\end{lemma}

\begin{proof}
    Consider the determinant $\delta:=u_{1}v_{2}-u_{2}v_{1}$, which is non-zero because $u$ and $v$ are linearly independent. Computing the cross product between $w(u)$ and $w(v)$ gives
    \begin{align*}
         w(u)\times w(v)&=\bigl(-u_1u_2v_1^2+v_1v_2u_1^2, u_2^2v_1^2-u_1^2v_2^2,u_2^2v_1v_2-v_2^2u_1u_2\bigr)\\
         &=\bigl(u_1v_1\delta, -(u_1v_2+u_2v_1)\delta, -u_2v_2\delta\bigr)\\
         &=:(a_0,b_0,c_0).
    \end{align*}
    Applying the identity $(X+Y)^{2}-4XY=(X-Y)^{2}$ to $X=u_{1}v_{2}$ and $Y=u_{2}v_{1}$ yields
    \begin{equation}\label{eq:disc-delta}
        b_0^2+4a_0c_0 = \delta^2 \bigl[(u_1v_2+u_2v_1)^2-4u_1u_2v_1v_2\bigr] = \delta^{2}(u_1v_2 - u_{2}v_{1})^2 = \delta^4>0.
    \end{equation}
    In particular $(a_{0},b_{0},c_{0})\neq(0,0,0)$, that is, $w(u)\times w(v)\neq0$. Therefore, $w(u)$ and $w(v)$ are linearly independent. Consequently $\operatorname{span}\{w(u),w(v)\}$ is a subspace of $\mathbb{R}^{3}$ with dimension 2, and its orthogonal
    complement is spanned by $(a_0,b_0,c_0)$.

    Let now $(a,b,c)\neq(0,0,0)$ be orthogonal to $w(u)$ and $w(v)$. In this case, note that $(a,b,c)=k(a_0,b_0,c_0)$ for some $k\ne 0$. Then, from \eqref{eq:disc-delta} we conclude
    \[
        b^{2}+4ac = k^2\bigl(b_0^2+4a_0c_0\bigr) = k^2\delta^{4}>0.
    \]
\end{proof}

\begin{theorem}\label{teo:r2char}
    Let $\Omega\subseteq\mathbb R^2$ be open and convex, and let $\Phi\colon\Omega\to\mathbb R^2$ be a $\mathcal C^1$ function. Consider $q\colon\Omega\to\mathbb R^3$ defined by $q(x,y)=(\varphi_{12}(x,y),\varphi_{11}(x,y)-\varphi_{22}(x,y),\varphi_{21}(x,y))$. The following are equivalent:
    \begin{enumerate}
        \item[(i)] $\Phi$ is diagonalizable,
        \item[(ii)] there exist $\rho\colon\Omega\to\mathbb R$ and $(a,b,c)\in\mathbb R^3$ such that $b^2+4ac>0$ and $q(x,y)=\rho(x,y)(a,b,c)$ for all $(x,y)\in\Omega$.
    \end{enumerate}
\end{theorem}
\begin{proof}
    Define $w\colon\mathbb R^2\to\mathbb R^3$ as in Lemma \ref{lemma:algebra}. Then, for $q=(\varphi_{12},\,\varphi_{11}-\varphi_{22},\,\varphi_{21})$, we can write equation
    \eqref{eq:latlip_quadratic_b} of Lemma \ref{lemma:latlip_quadratic_eq} as
    \[
        \bigl\langle q,\,w(\beta_1,\beta_2)\bigr\rangle
        =-\varphi_{21}\beta_1^2+(\varphi_{11}-\varphi_{22})\beta_1\beta_2+\varphi_{12}\beta_2^2=0 .
    \]
    Accordingly, given $(a,b,c)\in\mathbb{R}^3$ we write
    \[
        \tilde Q(\beta_1,\beta_2):=\bigl\langle (a,b,c),\,w(\beta_1,\beta_2)\bigr\rangle
        =-c\beta_1^2+b\beta_1\beta_2+a\beta_2^2 .
    \]

    Assume first (ii), that is, $q(x,y)=\rho(x,y)(a,b,c)$ on $\Omega$ with $b^2+4ac>0$. Note that then $(a,b,c)\neq(0,0,0)$. We first exhibit two linearly independent zeros of $\tilde Q$.

    If $a=0$, then $\tilde Q(\beta_1,\beta_2)=\beta_1(-c\beta_1+b\beta_2)$, so $(0,1)$ and $(b,c)$ are zeros of $\tilde Q$. Since $0<b^2+4ac=b^2$, we have $b\neq0$, hence $(0,1)$ and $(b,c)$ are linearly independent.

    If $a\neq0$, the equation $am^2+bm-c=0$ has two different real solutions $m_1\neq m_2$, because its discriminant is $b^2+4ac>0$. By Cardano-Vieta formulas, $m_1+m_2=-b/a$ and $m_1m_2=-c/a$. Therefore,
    \begin{align*}
        \tilde Q(\beta_1,\beta_2)
        &=a\Bigl(\beta_2^2+\tfrac{b}{a}\beta_1\beta_2-\tfrac{c}{a}\beta_1^2\Bigr)\\
        &=a\bigl(\beta_2^2-m_1\beta_1\beta_2-m_2\beta_1\beta_2+m_1m_2\beta_1^2\bigr)\\
        &=a\bigl(\beta_2(\beta_2-m_1\beta_1)-m_2\beta_1(\beta_2-m_1\beta_1)\bigr)\\
        &=a(\beta_2-m_1\beta_1)(\beta_2-m_2\beta_1).
    \end{align*}
    Thus $(1,m_1)$ and $(1,m_2)$ are zeros of $\tilde Q$, and they are linearly independent because $m_1\neq m_2$.

    In either case we have obtained two linearly independent vectors $\beta^1,\beta^2\in\mathbb{R}^2$, not depending on $(x,y)$, with $\tilde Q(\beta^i)=0$. Consequently, for every $(x,y)\in\Omega$ and $i\in\{1,2\}$,
    \[
        \bigl\langle q(x,y),w(\beta^i)\bigr\rangle
        =\rho(x,y)\bigl\langle (a,b,c),w(\beta^i)\bigr\rangle=\rho(x,y)\,\tilde Q(\beta^i)=0.
    \]
    That is, $\beta^1$ and $\beta^2$ solve \eqref{eq:latlip_quadratic_b} at every point of $\Omega$. By Lemma \ref{lemma:latlip_quadratic_eq}, $\Phi$ is diagonalizable, and it is diagonal with respect to the basis $\{\beta^1,\beta^2\}$. This proves (i).

    Suppose now (i). By Lemma \ref{lemma:latlip_quadratic_eq} there exist two linearly independent vectors $u=(u_1,u_2)$ and $v=(v_1,v_2)$, not depending on $(x,y)$, such that
    \begin{equation*}
        \bigl\langle q(x,y),w(u)\bigr\rangle=\bigl\langle q(x,y),w(v)\bigr\rangle=0
        \qquad\text{for all }(x,y)\in\Omega.
    \end{equation*}
    From this it follows
    \[
        \langle q(x,y),h\rangle=0, \quad \text{for all }(x,y)\in\Omega \text{ and } h\in H:=\operatorname{span}\{w(u),w(v)\}
    \]
    That is,
    \begin{equation}\label{eq:orth-W}
        q(x,y)\in H^\perp \quad \text{for all }(x,y)\in\Omega
    \end{equation}
    
    Since $u$ and $v$ are linearly independent, by Lemma \ref{lemma:algebra} we have that $w(u)$ and $w(v)$ are also linearly independent in $\mathbb R^3$. It follows $\operatorname{dim} H=2$, thus $\operatorname{dim} H^\perp=1$. Hence, there exists a non-zero vector $(a,b,c)\in\mathbb R^3$ (orthogonal to $w(u)$ and $w(v)$) such that $H^\perp=\operatorname{span}\{(a,b,c)\}$. Consequently, by \eqref{eq:orth-W} we get $q(x,y)=\rho(x,y)(a,b,c)$ for all $(x,y)\in\Omega$ and some $\rho(x,y)\in\mathbb R$. Moreover, this vector $(a,b,c)$ satisfies $b^2+4ac>0$ also by Lemma \ref{lemma:algebra}, completing the proof.
\end{proof}

\begin{remark}\label{remark:method}
    From Theorem \ref{teo:r2char} it follows the following procedure for determining whether a nonlinear mapping $\Phi$ is diagonalizable, and for constructing a corresponding diagonalizing basis.
    \begin{enumerate}
        \item Verify that $\Phi$ is of class $\mathcal C^{1}$ on an open and convex set $\Omega\subseteq\mathbb R^{2}$.
        \item If $q\equiv0$, then $\Phi$ is already diagonal (see Example \ref{ex:cauchy-riemann} below).
        \item In other case, fix a point $(x_{0},y_{0})\in\Omega$ such that $q(x_0,y_0)\neq 0$ and compute the cross product $q(x,y)\times q(x_{0},y_{0})$.
        \item If $q(x,y)\times q(x_{0},y_{0})=0$ for all $(x,y)\in\Omega$, then check whether $(a,b,c)=q(x_{0},y_{0})$ satisfies the condition $b^{2}+4ac>0$.
        \item If the above conditions hold, a basis that diagonalizes $\Phi$ can be chosen as follows:
        \begin{itemize}
            \item If $a\neq 0$, solve the quadratic equation $a m^{2}+b m-c=0$ for $m\in\mathbb R$. There must exist two distinct real solutions $m_{1}\neq m_{2}$. Then one may take $\mathcal B=\{(1,m_{1}),(1,m_{2})\}$.
            \item If $a=0$, one may take $\mathcal B=\{(0,1),(b,c)\}$.
        \end{itemize}
    \end{enumerate}
\end{remark}

Let us present a comprehensive example in which we study the diagonalization problem of a nonlinear function $\Phi\colon\mathbb R^2\to\mathbb R^2$ and find the corresponding basis.

\begin{example}\label{ex:diagex}
Consider the $\mathcal C^1$ function
\begin{equation}\label{eq:phiex}
    \Phi(x,y)=(2\exp(-x+y)-x+2y, \exp(-x+y)-x+2y),
\end{equation}
defined on $\mathbb R^2$. Its Jacobian matrix is $$D\Phi(x,y)=\begin{pmatrix}
    \varphi_{11}(x,y) & \varphi_{12}(x,y)\\
    \varphi_{21}(x,y) & \varphi_{22}(x,y)
\end{pmatrix}=\begin{pmatrix}
    -2\exp(-x+y)-1 & 2\exp(-x+y)+2\\
    -\exp(-x+y) -1 & \exp(-x+y) +2
\end{pmatrix}.$$

Therefore,
\begin{align*}
    q(x,y)&=(\varphi_{12}(x,y),\varphi_{11}(x,y)-\varphi_{22}(x,y),\varphi_{21}(x,y))\\&=(2\exp(-x+y)+2,-3\exp(-x+y)-3,-\exp(-x+y)-1)\\
    &=(\exp(-x+y)+1)(2,-3,-1).
\end{align*}
That is, we can express $q$ as $q(x,y)=\rho(x,y)(a,b,c)$ for all $(x,y)\in\Omega$, where $\rho(x,y)=\exp(-x+y)+1$ and $(a,b,c)=(2,-3,-1)$. Furthermore, $b^2+4ac=9-8>0$. Consequently, by Theorem \ref{teo:r2char} we deduce that $\Phi$ is diagonalizable. 

Furthermore, by Remark \ref{remark:method}, we must solve $2m^2-3m+1=0$ for $m$ in order to get a basis, since $a\neq 0$. The two real solutions are $m_1=1$ and $m_2=1/2$, hence $\overline{\mathcal B}=\{(1,1),(1,1/2)\}$ is a basis in which $\Phi$ is diagonal. 

Finally, we compute the diagonal expression of $\Phi$ in the equivalent, more convenient basis $\mathcal B=\{(1,1),(2,1)\}$. Since
\[
    P:=P_{\mathcal B, \mathcal C}=\begin{pmatrix}
        1 &2\\
        1&1
    \end{pmatrix},
\]
then
\[
    P_{\mathcal C, \mathcal B}=P^{-1}=\begin{pmatrix}
        -1 &2\\
        1&-1
    \end{pmatrix}.
\]
Therefore,
\[
    \Psi(s,t)=(P^{-1}\circ\Phi\circ P)(s,t)=P^{-1}\Phi(s+2t,s+t)=P^{-1}\begin{pmatrix}
        2\exp(-t)+s\\
        \exp(-t)+s
    \end{pmatrix}=\begin{pmatrix}
        s\\
        \exp(-t)
    \end{pmatrix}.
\]
\end{example}

The convexity assumption on $\Omega$ is needed for only one of the two implications of Theorem \ref{teo:r2char}. Indeed, recall from Remark \ref{remark:convex-domain} that, for an arbitrary open set $\Omega\subseteq\mathbb{R}^n$ and a $\mathcal{C}^1$ map $\Phi$, the implication
\[
    \Phi \text{ diagonal} \Rightarrow D\Phi(x) \text{ diagonal for every }
    x\in\Omega
\]
always holds, whereas its converse may fail and is guaranteed by the convexity of $\Omega$. Accordingly, this is translated to the equivalence established in the proof of Lemma \ref{lemma:latlip_quadratic_eq} between the diagonality of $\Phi$ in a basis $\mathcal{B}$ and the differential conditions \eqref{eq:latlip_lemma_partials_0_proof}. The direct implication of Lemma \ref{lemma:latlip_quadratic_eq} is immediate and requires no assumption on $\Omega$, while the converse implication is precisely the converse part of Remark~\ref{remark:convex-domain}, applied not to $\Omega$ but to the transformed domain $\tilde{\Omega}=P_{\mathcal{C},\mathcal{B}}\,\Omega$. Therefore, the implication (i) $\Rightarrow$ (ii) of Theorem \ref{teo:r2char} is valid on every open set $\Omega\subseteq\mathbb{R}^2$, convex or not. Only the implication (ii) $\Rightarrow$ (i) uses the convexity of $\Omega$, and it does so through the converse part of Remark \ref{remark:convex-domain} on $\tilde{\Omega}$. 

We provide the following example to show that the assumption of $\Omega$ convex in Theorem \ref{teo:r2char} cannot be omitted, nor weakened to connectedness.

\begin{example}
    Consider $\Omega=\mathbb R^2\setminus\{(x,0):x\ge0\}$, which is open, connected but not convex. Define $g\colon\Omega\to\mathbb R$ by
    \[
        g(x,y)=\begin{cases}
            \exp(-1/x) & \text{if } x,y>0,\\
            0 & \text{elsewhere}.
        \end{cases}
    \]
    Then define $\Phi\colon\Omega\to\mathbb R^2$ by $\Phi(x,y)=(g(x,y),y)$. Note that $\Phi$ is $C^\infty$ on $\Omega$, and $q(x,y)=\rho(x,y)(a,b,c)$ with $\rho=\partial g/\partial x-1$ and $(a,b,c)=(0,1,0)$. Since $b^2+4ac=1>0$, then $\Phi$ satisfies (ii) of Theorem \ref{teo:r2char}. However, $\Phi$ is not diagonalizable. Suppose that it were diagonal with respect to some basis $\{v_{1},v_{2}\}$ of $\mathbb R^{2}$. By the direct implication of Lemma \ref{lemma:latlip_quadratic_eq}, which requires no assumption on $\Omega$, both $v_{1}$ and $v_{2}$ would solve \eqref{eq:latlip_quadratic_b} at every point of $\Omega$. Now, at any point with $y<0$ we have $\partial g/\partial x=0$, hence $\rho=-1\neq0$, and there equation \eqref{eq:latlip_quadratic_b} reduces to $\rho\,\beta_{1}\beta_{2}=0$. That is, $\beta_{1}\beta_{2}=0$, whose set of solutions is the union of the two coordinate axes. Consequently $v_{1}$ and $v_{2}$ would be non-zero multiples of $e_{1}$ and $e_{2}$, in some order, and since scaling and reordering the vectors of a basis does not affect diagonality, $\Phi$ would be diagonal with respect to the canonical basis. This is not the case, because $\Phi_{1}(1,1)=\exp(-1)\neq0=\Phi_{1}(1,-1)$, so that $\Phi_{1}$ is not a function of $x$ alone.
\end{example}

For an arbitrary open set, condition (ii) remains necessary although it is no longer sufficient. Nevertheless, this reduces the search to a single candidate basis (except for scaling and order). This observation leads to the following result.

\begin{corollary}\label{cor:nonconvex}
    Let $\Omega\subseteq\mathbb R^2$ be an open set, not necessarily convex, and let $\Phi\colon\Omega\to\mathbb R^2$ be a $\mathcal C^1$ function.
    \begin{enumerate}
        \item[(a)] If $\Phi$ is diagonalizable, then condition (ii) of Theorem \ref{teo:r2char} holds.
        \item[(b)] Assume that (ii) holds for some $\rho$ not identically zero, and let $\mathcal B$ be the basis associated with $(a,b,c)$ as in Remark \ref{remark:method}. Then $\mathcal B$ is, up to scaling and ordering of its vectors, the only basis with respect to which $\Phi$ can be diagonal. Consequently, $\Phi$ is diagonalizable if and only if $P_{\mathcal B,\mathcal C}^{-1}\circ\Phi\circ P_{\mathcal B,\mathcal C}$ is diagonal.
    \end{enumerate}
\end{corollary}

\begin{proof}
    Part (a) follows from the proof of Theorem \ref{teo:r2char}, since the implication (i) $\Rightarrow$ (ii) uses only the direct implication of Lemma \ref{lemma:latlip_quadratic_eq} (which requires no assumption on $\Omega$) together with Lemma \ref{lemma:algebra} (which is purely algebraic).
    
    For (b), fix $(x_0,y_0)\in\Omega$ with $\rho(x_0,y_0)\neq0$. Again by the direct implication of Lemma \ref{lemma:latlip_quadratic_eq}, the vectors of any basis with respect to which $\Phi$ is diagonal are zeros of $\tilde Q(\beta_1,\beta_2)=\langle (a,b,c),w(\beta_1,\beta_2)\rangle$. As shown in the proof of Theorem \ref{teo:r2char}, the condition $b^2+4ac>0$ implies that the zero set of $\tilde Q$ is the union of the two lines spanned by the vectors of $\mathcal B$, which are linearly independent. Hence any such basis is obtained from $\mathcal B$ by scaling and reordering, and neither operation affects diagonality.
\end{proof}

Some particular cases of Theorem \ref{teo:r2char} include the following examples, illustrating well-known structural properties of certain nonlinear maps.

\begin{example}\label{ex:cauchy-riemann}
    Let $\Omega$ be open and convex, and let $\Phi\colon\Omega\subseteq\mathbb R^2\to\mathbb R^2$ be a $\mathcal C^1$ function such that
    \begin{equation}\label{eq:CR}
        \varphi_{11}=\varphi_{22}\qquad\text{and}\qquad\varphi_{12}=-\varphi_{21}.
    \end{equation}
    In this case, $q(x,y)=(\varphi_{12}(x,y),0,-\varphi_{12}(x,y))$. 

    Suppose that $\varphi_{12}\not\equiv 0$. Therefore, for any factorization $q(x,y)=\rho(x,y)(a,b,c)$ we get $(a,b,c)=k(1,0,-1)$ for some $k\neq 0$. Since $b^2+4ac=-4k^2<0$, then we conclude that $\Phi$ is not diagonalizable

    Consequently, under \eqref{eq:CR}, $\Phi$ is diagonalizable if and only if $\varphi_{12}=\varphi_{21}\equiv0$. In that case $\Phi_1$ depends only on $x$, $\Phi_2$ depends only on $y$, and $\Phi_1'(x)=\Phi_2'(y)$ for every $(x,y)\in\Omega$. Since $\Omega$ is convex, both derivatives are the same constant $\alpha$, so that $\Phi(x,y)=\alpha(x,y)+(\beta_1,\beta_2)$. Therefore, $\Phi$ is diagonal with respect to every basis of $\mathbb R^2$.
    
    This condition also follows from Theorem \ref{teo:r2char}, since \eqref{eq:CR} together with $\varphi_{12}=\varphi_{21}\equiv0$ yields $q\equiv 0$. Thus, for any $(a,b,c)\in\mathbb R^3$ such that $b^2+4ac>0$ we can write $q(x,y)=\rho(x,y)(a,b,c)$ choosing $\rho(x,y)=0$. For instance, for
    every $m_1,m_2\in\mathbb R$ with $m_1\neq m_2$ put $a=1, b=-(m_1+m_2)$ and $c=-m_1m_2$. Then, $b^2+4ac=(m_1-m_2)^2>0$ and we get the basis $\mathcal B=\{(1,m_1),(1,m_2)\}$.

    This example is notable because \eqref{eq:CR} are the Cauchy--Riemann equations, characterising the holomorphic functions on $\mathbb C$. Consequently, the only holomorphic functions that are diagonalizable (in the real sense) are the maps $f(z)=\alpha z+\beta$ with $\alpha\in\mathbb R$ and $\beta\in\mathbb C$. 
    
    Observe also that the linear case associated to \eqref{eq:CR} is the classical diagonalization problem for the matrix
    \begin{equation}
        M=\begin{pmatrix}\alpha&\beta\\-\beta&\alpha\end{pmatrix},
    \end{equation}
    a well-known example of a matrix which is not diagonalizable over $\mathbb R$ when $\beta\neq0$, and diagonal with respect to every basis when $\beta=0$.
\end{example}

\begin{example}
	Another important case is when
	\begin{equation}
		\label{eq:t_phi_A1_B0}
		\varphi_{11} = \varphi_{22} \quad \text{and} \quad \varphi_{12} = \varphi_{21}\not\equiv0.
	\end{equation}
	This implies that the Jacobian matrix of $\Phi$ is symmetric. If $\Omega$ is convex, and in particular simply connected, then there exists $f : \Omega \subseteq \mathbb R^2 \to \mathbb R$ being $\mathcal C^2$ such that $\Phi = \nabla f$.
	Moreover, as $\frac{\partial^2 f}{\partial x^2} = \frac{\partial^2 f}{\partial y^2}$, we have that there exist $g \in\mathcal C^2(\Omega_1)$ and $h \in\mathcal C^2(\Omega_2)$ with $\Omega_1 = \{ x-y : (x,y) \in \Omega \}$ and $\Omega_2 = \{ x+y : (x,y) \in \Omega \}$ such that
	\begin{equation}
		f(x,y) = h(x+y) + g(x-y).
	\end{equation}
	The mapping $\Phi$ is then diagonal with respect to the vectors $(1,1)$ and $(1,-1)$.
	
	The same result can be seen as a consequence of Theorem \ref{teo:r2char}.
	As $q(x,y)=\varphi_{12}(x,y)(1,0,1)$ and $(a,b,c)=(1,0,1)$ satisfies $b^2+4ac=4>0$, then $\Phi$ is diagonalizable. Moreover, solving $am^2+bm-c=m^2-1=0$ we get $m=\pm 1$, thus by Remark \ref{remark:method} we also obtain the basis
    $
    \{ ( 1, 1 ), (1, -1) \}.
    $
	
\end{example}

\section{Applications} 
\label{appl}

As in the linear case, finding a diagonal form for a mapping by means of a change of basis provides an interesting advantage for explicit analytical computations of certain fundamental processes for applications. In this section, we proceed to demonstrate three such applications.

\subsection{High-order composition and the exponential of non-linear functions}

Diagonalization for a general function $\Phi\colon\Omega\subseteq\mathbb R^n\to\Omega$ 
is a powerful tool to compute high-order compositions, since $\Phi^{(k)}=P\circ\Psi^{(k)}\circ P^{-1}$ where $P$ is the matrix of the isomorphism to get the diagonal form $\Psi$.
Since $\Psi$ is diagonal, the problem is reduced to compute $k$ high-order composition but from simpler functions of one variable. In this direction, this section provides the first application of our results. 

\begin{example}\label{ex:compex}
Consider $\Omega=\{(x,y)\in\mathbb R^2:y\neq x\}$ and $\Phi\colon\Omega\to\mathbb R^2$ defined as
\begin{equation}\label{eq:ex}
    \Phi(x,y)= \left( 4x-2y+\frac{1}{y-x},4x-2y+\frac{2}{y-x} \right)
\end{equation}
It can be checked that $\Omega$ is open and $\Phi(\Omega)\subset\Omega$, but $\Omega$ is not convex. However, note that
\[
    D\Phi(x,y)=\begin{pmatrix}
        4+(y-x)^{-2} & -2-(y-x)^{-2}\\
        4+2(y-x)^{-2} & -2-2(y-x)^{-2}
    \end{pmatrix},
\]
hence
\[
    q(x,y)=(-2-(y-x)^{-2},6+3(y-x)^{-2},4+2(y-x)^{-2})=[2+(y-x)^{-2}](-1,3,2).
\]
Therefore, (ii) of Theorem \ref{teo:r2char} holds for $\rho(x,y)=2+(y-x)^{-2}$ (not identically zero on $\Omega$) and $(a,b,c)=(-1,3,2)$ satisfying $b^2+4ac=1>0$. Furthermore, the associated basis to $(a,b,c)$ according to Remark \ref{remark:method} is $\mathcal B=\{(1,1),(1,2)\}$. Therefore, by Corollary \ref{cor:nonconvex} it follows that $\Phi$ is diagonalizable if and only if $P_{\mathcal B,\mathcal C}^{-1}\circ\Phi\circ P_{\mathcal B,\mathcal C}$ is diagonal. A direct computation gives
\[
    (P_{\mathcal B,\mathcal C}^{-1}\circ\Phi\circ P_{\mathcal B,\mathcal C})(s,t)=P_{\mathcal B,\mathcal C}^{-1}\Phi(s+t,s+2t)=\begin{pmatrix}
        2&-1\\ -1&1
    \end{pmatrix}\begin{pmatrix}
        2s + t^{-1}\\
        2s+2t^{-1}
    \end{pmatrix}=\begin{pmatrix}
        2s\\ t^{-1}
    \end{pmatrix}
\]
Since this expression is diagonal we conclude that $\Phi$ is diagonalizable in the basis $\mathcal B$, with diagonal expression $\Psi(s,t)=(2s,t^{-1})$ defined on $P_{\mathcal B,\mathcal C}^{-1}(\Omega)=\{(s,t)\in\mathbb R^2:t\neq 0\}$.

The $k$-fold composition of $\Psi$ can be easily derived as
$$\Psi^{(k)}(s,t)=(2^ks, t^{(-1)^k}),$$
and therefore
    \begin{align*}
        \Phi^{(k)}(x,y)&=(P\circ\Psi^{(k)}\circ P^{-1})(x,y)\\ &=(P\circ\Psi^{(k)})(2x-y,-x+y)\\&=P\begin{pmatrix}
        2^k(2x-y)\\  (y-x)^{(-1)^k}
    \end{pmatrix}\\&=\begin{pmatrix}
        2^k(2x-y)+(y-x)^{(-1)^k}\\
        2^k(2x-y)+2(y-x)^{(-1)^k}
    \end{pmatrix}.
    \end{align*}
\end{example}

As a particular case, we find the computation of the exponential function. 
Given a matrix $A\in\mathbb R^{n\times n}$, recall that its exponential can be defined as the power series $$e^A=\sum_{k=0}^\infty \frac{1}{k!}A^k.$$
In this way, the exponential of a linear function can be defined. Now we are interested in defining an equivalent concept for a general function $\Phi\colon\Omega\subseteq\mathbb R^n\to\Omega$. We define $e^\Phi\colon\Omega\subseteq\mathbb R^n\to\mathbb R^n$ as $$e^\Phi(\cdot)=\sum_{k=0}^\infty \frac{1}{k!}\Phi^{(k)}(\cdot),$$
where $\Phi^{(k)}$ denotes the $k$-fold composition. We will also consider the function $E_\Phi\colon\Omega\times\mathbb R\to\mathbb R^n$ defined as
\begin{equation}\label{eq:expxi}
    E_{\Phi}(x,\xi)=\sum_{k=0}^\infty \frac{\xi^k}{k!}\Phi^{(k)}(x),
\end{equation}
whenever it converges. Clearly, $E_\Phi(x,0) = x$ due to the usual convention $\Phi^{(0)}=Id$, and also yields $E_\Phi(x,1) = e^\Phi(x)$. Some basic examples of this function are the following, where $e^t$ denotes the usual exponential of real numbers.
\begin{enumerate}[(i)]
    \item For the identity function $\Phi(x) = x$, we get $E_\Phi(x,\xi) = e^\xi x$.
    \item More in general, when $\Phi(x) = \mu x$ for $\mu \in \mathbb R$, yields $E_\Phi(x,\xi) = e^{\mu \xi} x$.
    \item If $\Phi(x) = x + x_0$ for $x_0 \in \mathbb R^n$, then $E_\Phi(x,\xi) = e^{\xi} (x+\xi x_0)$.
    \item For a linear map $\Phi(x) = A x$, where $A$ is a real $n \times n$ matrix, our definition of the exponential function matches the well-known matrix exponential: $E_\Phi(x,\xi) = e^{A \xi} x$.
\end{enumerate}

The following results will cover the study of the convergence of this series.

\begin{proposition}\label{prop:radius}
    Let $\Omega\subseteq\mathbb R^{n}$ and let $\Phi\colon\Omega\to\Omega$. Fix $x\in\Omega$ and consider the power series \eqref{eq:expxi} in the variable $\xi$ with coefficients $a_{k}(x):=\Phi^{(k)}(x)/k!\in\mathbb R^{n}$. Let
    \[
        R(x):=\frac{1}{\displaystyle\limsup_{k}\,\|a_{k}(x)\|^{1/k}}\ \in[0,+\infty],
    \]
    with the conventions $1/0=+\infty$ and $1/(+\infty)=0$. Equivalently, whenever $\Phi^{(k)}(x)\neq0$ for every $k$,
    \[
        R(x)=\liminf_{k}\left(\frac{k!}{\|\Phi^{(k)}(x)\|}\right)^{1/k}.
    \]
    Then,
    \begin{enumerate}
        \item[(i)] $E_{\Phi}(x,\xi)$ converges absolutely for every $\xi\in\mathbb R$ with $|\xi|<R(x)$, and diverges for every $\xi\in\mathbb R$ with $|\xi|>R(x)$.
        \item[(ii)] For every $0\leq r<R(x)$, the series converges absolutely and uniformly on $\{\xi\in\mathbb R:|\xi|\leq r\}$.
    \end{enumerate}
\end{proposition}

\begin{proof}
    Consider the series of non-negative real numbers $\sum_{k}\|a_{k}(x)\|\,|\xi|^{k}$, whose general term satisfies
    \[
        \bigl(\|a_{k}(x)\|\,|\xi|^{k}\bigr)^{1/k}=|\xi|\cdot\|a_{k}(x)\|^{1/k},
    \]
    hence
    \[
        \limsup_{k}\bigl(\|a_{k}(x)\|\,|\xi|^{k}\bigr)^{1/k}=\frac{|\xi|}{R(x)}.
    \]
    If $|\xi|<R(x)$ this limit is smaller than $1$, and the root test \cite[Theorem 3.33]{babyrudin} gives the convergence of $\sum_{k}\|a_{k}(x)\|\,|\xi|^{k}$. That is, the series $\sum_{k}a_{k}(x)\xi^{k}$ converges absolutely, and since $\mathbb R^{n}$ is complete, it converges. If $|\xi|>R(x)$ the limit is larger than $1$, so the general term does not tend to $0$ and the series diverges. This proves (i)

    On the other hand, let $0\leq r<R(x)$ and put $M_{k}:=\|a_{k}(x)\|\,r^{k}$. By (i) applied to $\xi=r$ the series $\sum_{k}M_{k}$ converges, and $\|a_{k}(x)\xi^{k}\|\leq M_{k}$ for every $\xi$ with $|\xi|\leq r$ and every $k$. The Weierstrass $M$-criterion \cite[Theorem 7.10]{babyrudin} then yields the absolute and uniform convergence of the series on $\{|\xi|\leq r\}$.
\end{proof}

Let us now consider the Lipschitz condition for $\Phi$ as well as $\Phi(0)=0$ in order to show that, under particular requirements, we can obtain better convergence results. Note, however, that the example at the beginning of this section does not satisfy these properties.

\begin{proposition}\label{prop:lipschitz}
    Let $\Omega\subseteq\mathbb R^{n}$ with $0\in\Omega$, and let $\Phi\colon\Omega\to\Omega$ be a Lipschitz function such that $\Phi(0)=0$. Then,
    \begin{enumerate}
        \item[(i)] $R(x)=+\infty$ for every $x\in\Omega$. Consequently, $E_{\Phi}$ is defined on the whole of $\Omega\times\mathbb R$.
        \item[(ii)] The series \eqref{eq:expxi} converges absolutely and uniformly on $K\times J$, for every compact subset $K\subseteq\Omega$ and every compact interval $J\subseteq\mathbb R$. 
        \item[(iii)] For each $\xi\in\mathbb R$ the function $E_{\Phi(\cdot,\xi)}$ is Lipschitz, with
        \[
            \|E_{\Phi}(\cdot,\xi)\|_{Lip}\leq e^{|\xi|\,\|\Phi\|_{Lip}}.
        \]
    \end{enumerate}
\end{proposition}

\begin{proof}
    Recall that the Lipschitz seminorm of $f\colon\Omega\to\mathbb R^n$ is
    \[
        \|f\|_{Lip}=\sup\left\{\frac{\|f(x)-f(y)\|}{\|x-y\|}:x,y\in\Omega, \, x\neq y\right\},
    \]
    which is finite for Lipschitz functions. This seminorm is submultiplicative under composition, that is, $\|f\circ g\|_{Lip}\leq\|f\|_{Lip}\|g\|_{Lip}$. By induction it follows $\|\Phi^{(k)}\|_{Lip}\leq\|\Phi\|_{Lip}^{k}$ for every $k\in\mathbb N$. Since $\Phi(0)=0$, also $\Phi^{(k)}(0)=0$ for every $k$, and therefore
    \[
        \|\Phi^{(k)}(x)\|=\|\Phi^{(k)}(x)-\Phi^{(k)}(0)\|\leq\|\Phi^{(k)}\|_{Lip}\|x\|
        \leq\|\Phi\|_{Lip}^{k}\,\|x\|,\qquad x\in\Omega,\ k\in\mathbb N .
    \]
    Consequently, the general term of \eqref{eq:expxi} is dominated as follows:
    \begin{equation}\label{eq:expbound}
        \left\|\frac{\xi^{k}}{k!}\,\Phi^{(k)}(x)\right\|
        \leq\frac{\bigl(|\xi|\,\|\Phi\|_{Lip}\bigr)^{k}}{k!}\,\|x\|,
        \qquad x\in\Omega,\ \xi\in\mathbb R,\ k\in\mathbb N .
    \end{equation}

    Fix $x\in\Omega$ and $\xi\in\mathbb R$. Summing \eqref{eq:expbound} over $k$ and using that the exponential series $\sum_{k}t^{k}/k!$ converges to $e^{t}$ for every $t\in\mathbb R$, we get
    \[
        \sum_{k=0}^{\infty}\left\|\frac{\xi^{k}}{k!}\,\Phi^{(k)}(x)\right\|
        \leq\|x\|\sum_{k=0}^{\infty}\frac{\bigl(|\xi|\,\|\Phi\|_{Lip}\bigr)^{k}}{k!}
        =\|x\|\,e^{|\xi|\,\|\Phi\|_{Lip}}<\infty .
    \]
    Therefore, \eqref{eq:expxi} converges absolutely at every $\xi\in\mathbb R$. Since, by Proposition \ref{prop:radius}, the series diverges whenever $|\xi|>R(x)$, this forces $R(x)=+\infty$ proving (i).

    On the other hand, fix a compact subset $K\subseteq\Omega$ and a compact interval $J\subseteq\mathbb R$. Put $U_{K}:=\sup_{x\in K}\|x\|<\infty$ and $R_{J}:=\max_{\xi\in J}|\xi|<\infty$. By \eqref{eq:expbound}, for every $(x,\xi)\in K\times J$ and every $k$ it holds
    \[
        \left\|\frac{\xi^{k}}{k!}\,\Phi^{(k)}(x)\right\|
        \leq\frac{\bigl(R_{J}\,\|\Phi\|_{Lip}\bigr)^{k}}{k!}\,U_{K}=:M_{k}^{K,J}.
    \]
    By the convergence of the exponential series again it follows $\sum_{k}M_{k}^{K,J}=U_{K}\,e^{R_{J}\|\Phi\|_{Lip}}<\infty$. Consequently, by the Weierstrass $M$-criterion \cite[Theorem 7.10]{babyrudin} we deduce the absolute and uniform convergence of \eqref{eq:expxi} on $K\times J$. This proves (ii).

    Finally, let $\xi\in\mathbb R$ and $x,y\in\Omega$.The submultiplicativity of the Lipschitz seminorm gives
    \[
        \|E_{\Phi}(x,\xi)-E_{\Phi}(y,\xi)\|
        \leq\sum_{k=0}^{\infty}\frac{|\xi|^{k}}{k!}\,\|\Phi^{(k)}(x)-\Phi^{(k)}(y)\|
        \leq\|x-y\|\sum_{k=0}^{\infty}\frac{\bigl(|\xi|\,\|\Phi\|_{Lip}\bigr)^{k}}{k!}
        =e^{|\xi|\,\|\Phi\|_{Lip}}\,\|x-y\| ,
    \]
    where the series converges by (i). Hence $E_{\Phi}(\cdot,\xi)$ is Lipschitz with $\|E_{\Phi}(\cdot,\xi)\|_{Lip}\leq e^{|\xi|\|\Phi\|_{Lip}}$.
\end{proof}

\begin{remark}\label{rem:nozero}
    The assumption $\Phi(0)=0$ in Proposition \ref{prop:lipschitz} together with $0\in\Omega$ is not necessary to get the previous result. Indeed, let $\Omega\subseteq\mathbb R^{n}$ and let $\Phi\colon\Omega\to\Omega$ be a Lipschitz function, with no further assumption. In particular, $\Phi$ need not have a fixed point. Write $L:=\|\Phi\|_{Lip}$ and fix any $x_{0}\in\Omega$. Using that $\|\Phi^{(j)}\|_{Lip}\leq L^{j}$ and $\Phi^{(j+1)}(x_{0})=\Phi^{(j)}(\Phi(x_{0}))$ we get
    \[
        \bigl\|\Phi^{(k)}(x_{0})-x_{0}\bigr\|
        \leq\sum_{j=0}^{k-1}\bigl\|\Phi^{(j)}(\Phi(x_{0}))-\Phi^{(j)}(x_{0})\bigr\|
        \leq\Bigl(\sum_{j=0}^{k-1}L^{j}\Bigr)\,\|\Phi(x_{0})-x_{0}\|.
    \]
    Therefore, for every $x\in\Omega$ and every $k\in\mathbb N$,
    \begin{equation}
        \begin{aligned}
            \|\Phi^{(k)}(x)\|&\leq\|\Phi^{(k)}(x)-\Phi^{(k)}(x_{0})\|+\|\Phi^{(k)}(x_{0})\|\\&
        \leq L^{k}\|x-x_{0}\|
        +\Bigl(\sum_{j=0}^{k-1}L^{j}\Bigr)\|\Phi(x_{0})-x_{0}\|+\|x_{0}\| .
        \end{aligned}
        \label{eq:boundnofixed}
    \end{equation}

    Let $K\subseteq\Omega$ be compact and let $J\subseteq\mathbb R$ be a compact interval. Put $D_{K}:=\sup_{x\in K}\|x-x_{0}\|<\infty$ and $R_{J}:=\max_{\xi\in J}|\xi|$. From \eqref{eq:boundnofixed} we can bound the general term of \eqref{eq:expxi} on $K\times J$ by
    \[
        M_{k}^{K,J}:=\frac{R_{J}^{k}}{k!}\left(L^{k}D_{K}
        +\Bigl(\sum_{j=0}^{k-1}L^{j}\Bigr)\|\Phi(x_{0})-x_{0}\|+\|x_{0}\|\right).
    \]
    If $L\neq1$, then $\sum_{j=0}^{k-1}L^{j}=(1-L^{k})/(1-L)$ and
    \[
        \sum_{k=0}^{\infty}M_{k}^{K,J}
        =D_{K}\,e^{R_{J}L}
        +\frac{e^{R_{J}}-e^{R_{J}L}}{1-L}\,\|\Phi(x_{0})-x_{0}\|
        +\|x_{0}\|\,e^{R_{J}}<\infty.
    \]
    If $L=1$, then $\sum_{j=0}^{k-1}L^{j}=k$ and, since
    $\sum_{k}k\,R_{J}^{k}/k!=R_{J}e^{R_{J}}$, we get
    \[
        \sum_{k=0}^{\infty}M_{k}^{K,J}
        =\bigl(D_{K}+R_{J}\,\|\Phi(x_{0})-x_{0}\|+\|x_{0}\|\bigr)\,e^{R_{J}}<\infty .
    \]
    In both cases the Weierstrass $M$-criterion gives that \eqref{eq:expxi} converges absolutely and uniformly on $K\times J$. Since $J$ is an arbitrary compact interval, it follows that $R(x)=+\infty$ for every $x\in\Omega$, as in Proposition \ref{prop:lipschitz} (i). 
    
    Once the convergence is granted, the Lipschitz estimate of Proposition \ref{prop:lipschitz} (iii) holds trivially since its proof only uses the
    submultiplicativity of the Lipschitz seminorm, and not the value of $\Phi$ at any particular point.

    We nevertheless retain the condition $\Phi(0)=0$ in Proposition \ref{prop:lipschitz}, both because it is assumed in Section \ref{sec:diag}, see Proposition \ref{prop:eigenvalues}, and for the sake of simplicity to get the simpler Lipschitz constant $e^{|\xi|\,\|\Phi\|_{Lip}}$.
\end{remark}

The diagonalization tools developed in this paper facilitate the computation of the series $E_{\Phi}$, reducing it to $n$ scalar exponential series in one variable.

Let $\Psi\colon\tilde\Omega\to\tilde\Omega$ be a diagonal function, so that $\Psi(x)=(\Psi_{1}(x_{1}),\ldots,\Psi_{n}(x_{n}))$. Now note that
\[
    (\Psi\circ\Psi)(x)=\Psi\bigl(\Psi_{1}(x_{1}),\ldots,\Psi_{n}(x_{n})\bigr)
    =\bigl(\Psi_{1}^{(2)}(x_{1}),\ldots,\Psi_{n}^{(2)}(x_{n})\bigr),
\]
and by induction on $k$ we get $\Psi^{(k)}(x)=\bigl(\Psi_{1}^{(k)}(x_{1}),\ldots,\Psi_{n}^{(k)}(x_{n})\bigr)$ for every $k\in\mathbb N$. Since a series in $\mathbb R^{n}$ converges if and only if each of its $n$ coordinate series converges, and in that case its sum is the vector of the corresponding sums, we conclude that $E_{\Psi}(x,\xi)$ is defined if and only if $E_{\Psi_{i}}(x_{i},\xi)$ is defined for every $i\in\{1,\ldots,n\}$. In this case,
\[
    E_{\Psi}(x,\xi)=\sum_{k=0}^{\infty}\frac{\xi^{k}}{k!}
    \bigl(\Psi_{1}^{(k)}(x_{1}),\ldots,\Psi_{n}^{(k)}(x_{n})\bigr)
    =\bigl(E_{\Psi_{1}}(x_{1},\xi),\ldots,E_{\Psi_{n}}(x_{n},\xi)\bigr).
\]
In particular, the $i$-th coordinate of $E_{\Psi}(\cdot,\xi)$ depends only on $x_{i}$, thus $E_{\Psi}(\cdot,\xi)$ is again a diagonal function.

Suppose now that $\Phi$ is diagonalizable and write $\Phi=P\circ\Psi\circ P^{-1}$, where $\Psi$ is diagonal on $\tilde\Omega=P^{-1}\Omega$. It is immediate that $\Phi^{(k)}=P\circ\Psi^{(k)}\circ P^{-1}$ for every $k$. Moreover, since $P$ is continuous it maps convergent series into convergent series and commutes with their sums. Consequently, $E_{\Phi}(x,\xi)$ is defined if and only if $E_{\Psi}(P^{-1}x,\xi)$ is defined. In that case,
\[
    E_{\Phi}(x,\xi)=\sum_{k=0}^{\infty}\frac{\xi^{k}}{k!}\,(P\circ\Psi^{(k)}\circ P^{-1})(x)
    =P\sum_{k=0}^{\infty}\frac{\xi^{k}}{k!}\,\Psi^{(k)}(P^{-1}x)
    =P\,E_{\Psi}(P^{-1}x,\xi).
\]
Furthermore, let $R_\Phi(x)$ denote the radius of convergence of $E_\Phi(x,\cdot)$ for a fixed $x\in\Omega$, given in Proposition \ref{prop:radius}. Then, $R_{\Phi}(x)=R_{\Psi}(P^{-1}x)$.

Overall, the computation of $E_{\Phi}$ for a diagonalizable function $\Phi$ reduces to the computation of the $n$ scalar series $E_{\Psi_{i}}$, one for each coordinate, followed by the linear change of coordinates given by $P$.

An important feature of the introduced exponential expressions is that they are solutions to a wide range of differential equations characterised by functions $\Phi$ that may be not linear. Therefore, these differential equations can be easily solved as long as such functions $\Phi$ can be diagonalized.

\begin{proposition}\label{prop:solde}
    Let $\Omega\subseteq\mathbb R^{n}$ and let $\Phi\colon\Omega\to\Omega$. Then,
    \begin{enumerate}
        \item[(i)] $R(\Phi(x))=R(x)$ for every $x\in\Omega$. That is, the radius given in Proposition \ref{prop:radius} is constant along the orbits of $\Phi$.
        \item[(ii)] For every $x\in\Omega$ such that $R(x)>0$, the function $\xi\mapsto E_{\Phi}(x,\xi)$ is of class $\mathcal C^{\infty}$ on the interval $(-R(x),R(x))$. Moreover,
        \begin{equation}\label{eq:funcode}
            \frac{\partial}{\partial\xi}E_{\Phi}(x,\xi)=E_{\Phi}(\Phi(x),\xi), \qquad |\xi|<R(x).
        \end{equation}
    \end{enumerate}
    That is, $\Gamma(x,\xi):=E_{\Phi}(x,\xi)$ is a solution of the differential equation $\partial\Gamma(x,\xi)/\partial\xi=\Gamma(\Phi(x),\xi)$ on $\{(x,\xi):x\in\Omega,\ |\xi|<R(x)\}$, with $\Gamma(x,0)=x$.
\end{proposition}

\begin{proof}
    Fix $x\in\Omega$. By Proposition \ref{prop:radius}, $E_{\Phi}(x,\cdot)$ is the power series $\sum_{k\geq0}a_{k}(x)\,\xi^{k}$ with coefficients $a_{k}(x)=\Phi^{(k)}(x)/k!$ and radius of convergence $R(x)$. Its termwise derivative is $\sum_{j\geq0}(j+1)\,a_{j+1}(x)\,\xi^{j}$, and
    since $\Phi(x)\in\Omega$ and $\Phi^{(j+1)}=\Phi^{(j)}\circ\Phi$ we get
    \[
        (j+1)\,a_{j+1}(x)=(j+1)\,\frac{\Phi^{(j+1)}(x)}{(j+1)!}=\frac{\Phi^{(j+1)}(x)}{j!}
        =\frac{\Phi^{(j)}(\Phi(x))}{j!}=a_{j}(\Phi(x)),\qquad j\in\mathbb N .
    \]
    Therefore the termwise derivative of $E_{\Phi}(x,\cdot)$ is the power series $E_{\Phi}(\Phi(x),\cdot)$. Since a power series and its termwise derivative have the same radius of convergence \cite[Theorem 8.1]{babyrudin}, the radius of $E_{\Phi}(\Phi(x),\cdot)$ equals that of $E_{\Phi}(x,\cdot)$, which proves $(i)$.

    For $(ii)$, assume $R(x)>0$ and recall that the theorem on differentiation of power series \cite[Theorem 8.1]{babyrudin} gives that $E_{\Phi}(x,\cdot)$ is of class $\mathcal C^{\infty}$ on $(-R(x),R(x))$ and may be differentiated term by term there. Hence, for $|\xi|<R(x)$,
    \[
        \frac{\partial}{\partial\xi}E_{\Phi}(x,\xi)
        =\sum_{j=0}^{\infty}(j+1)\,a_{j+1}(x)\,\xi^{j}
        =\sum_{j=0}^{\infty}\frac{\xi^{j}}{j!}\,\Phi^{(j)}(\Phi(x))
        =E_{\Phi}(\Phi(x),\xi).
    \]
    Finally, $E_{\Phi}(x,0)=\Phi^{(0)}(x)=x$ by the convention $\Phi^{(0)}=\operatorname{Id}$.
\end{proof}

Note that (ii) of Proposition \ref{prop:solde} can be obtained over all $\mathbb R$ in an already studied case. Let $\Omega\subseteq\mathbb R^{n}$ with $0\in\Omega$, and let $\Phi\colon\Omega\to\Omega$ be a Lipschitz function such that $\Phi(0)=0$. By Proposition \ref{prop:lipschitz}(i) we have $R(x)=+\infty$ for every $x\in\Omega$, so the interval $(-R(x),R(x))$ of Proposition \ref{prop:solde} is $\mathbb R$. That is, $E_{\Phi}$ is defined on $\Omega\times\mathbb R$ for each $x\in\Omega$. Moreover, the function $E_{\Phi}(x,\cdot)$ is of class $\mathcal C^{\infty}$ on $\mathbb R$, and \eqref{eq:funcode} holds for every $(x,\xi)\in\Omega\times\mathbb R$. By Remark \ref{rem:nozero}, the same conclusion holds for any Lipschitz $\Phi\colon\Omega\to\Omega$, without assuming $0\in\Omega$ or $\Phi(0)=0$.

\begin{example}
    Consider the diagonal function $\Psi=(\Psi_{1},\Psi_{2})$ considered in Example \ref{ex:compex}, that is, $\Psi(s,t)=(2s,1/t)$ defined on $\tilde\Omega=\{(s,t)\in\mathbb R^{2}:t\neq0\}$. Since $1/t\neq0$, we have $\Psi(\tilde\Omega)\subseteq\tilde\Omega$, so the iterates $\Psi^{(k)}$ are defined on $\tilde\Omega$ and Proposition \ref{prop:solde} applies. Observe that $\Psi$ is not a Lipschitz function on $\tilde\Omega$, so we cannot get directly the radius of convergence of $E_\Psi$ and it has to be obtained from Proposition \ref{prop:radius}.

    For every $k\in\mathbb N$ we have
    \[
        \Psi_{1}^{(k)}(s)=2^{k}s,\qquad \Psi_{2}^{(k)}(t)=t^{(-1)^{k}} ,
    \]
    and hence, by the triangle inequality, $\|\Psi^{(k)}(s,t)\|\leq2^{k}|s|+m\leq2^{k}\bigl(|s|+m\bigr)$, where $m:=\max\{|t|,1/|t|\}$. Consequently, for every $\xi\in\mathbb R$,
    \[
        \sum_{k=0}^{\infty}\left\|\frac{\xi^{k}}{k!}\,\Psi^{(k)}(s,t)\right\| \leq\bigl(|s|+m\bigr)\sum_{k=0}^{\infty}\frac{\bigl(2|\xi|\bigr)^{k}}{k!} =\bigl(|s|+m\bigr)\,e^{2|\xi|}<\infty.
    \]
    Therefore, the series converges absolutely at every $\xi\in\mathbb R$ and thus $R(s,t)=+\infty$ for every $(s,t)\in\tilde\Omega$.

    Since $\Psi$ is diagonal then $E_{\Psi}\bigl((s,t),\xi\bigr)=\bigl(E_{\Psi_{1}}(s,\xi),\,E_{\Psi_{2}}(t,\xi)\bigr)$, where
    \begin{align*}
        E_{\Psi_{1}}(s,\xi)&=\sum_{k=0}^{\infty}\frac{\xi^{k}}{k!}\,\Psi_{1}^{(k)}(s)=s\sum_{k=0}^{\infty}\frac{(2\xi)^{k}}{k!}=e^{2\xi}s ,\\[2pt]
        E_{\Psi_{2}}(t,\xi)&=\sum_{k=0}^{\infty}\frac{\xi^{k}}{k!}\,\Psi_{2}^{(k)}(t) =\sum_{k=0}^{\infty}\frac{\xi^{k}}{k!}\,t^{(-1)^{k}} =t\sum_{k=0}^{\infty}\frac{\xi^{2k}}{(2k)!} +\frac{1}{t}\sum_{k=0}^{\infty}\frac{\xi^{2k+1}}{(2k+1)!}=t\cosh\xi+\frac{1}{t}\sinh\xi.
    \end{align*}
    Consequently,
    \[
        E_{\Psi}\bigl((s,t),\xi\bigr)
        =\left(e^{2\xi}s,\ t\cosh\xi+\frac{1}{t}\sinh\xi\right),
        \qquad (s,t)\in\tilde\Omega,\ \xi\in\mathbb R ,
    \]
    and in particular the exponential of $\Psi$ is
    \[
        e^{\Psi}(s,t)=E_{\Psi}\bigl((s,t),1\bigr)=\left(e^{2}s,\ t\cosh1+\frac{1}{t}\sinh1\right).
    \]

    Let us check identity \eqref{eq:funcode} directly in this case. On the one hand,
    \[
        \frac{\partial}{\partial\xi}E_{\Psi}\bigl((s,t),\xi\bigr) =\left(2e^{2\xi}s,\ t\sinh\xi+\frac{1}{t}\cosh\xi\right).
    \]
    On the other hand, evaluating $E_{\Psi}$ at the point $\Psi(s,t)=(2s,1/t)$ we get
    \[
        E_{\Psi}\bigl(\Psi(s,t),\xi\bigr) =E_{\Psi}\left(\left(2s,\tfrac{1}{t}\right),\xi\right) =\left(2s\,e^{2\xi},\ \frac{1}{t}\cosh\xi+t\sinh\xi\right),
    \]
    so that both expressions coincide, as expected.
\end{example}

\begin{example}
    Consider the differential equation $\partial\Gamma(x,\xi)/\partial\xi=\Gamma(\Phi(x),\xi)$ given by the function $\Phi$ defined in \eqref{eq:ex}. To compute the solution given by Proposition \ref{prop:solde}, recall that $\Phi$ is diagonalizable in the basis $\{(1,1),(1,2)\}$ with diagonal expression $\Psi(s,t)=(\Psi_1(s),\Psi_2(t))=(2s,1/t)$. Then, the desired solution is
    \begin{align*}
        E_\Phi(x,y,\xi)&=PE_{\Psi}(2x-y,-x+y,\xi)\\
        &=P\begin{pmatrix}
           e^{2\xi} (2x-y)\\
            (-x+y)\cosh(\xi)+\tfrac{1}{-x+y}\sinh(\xi)
        \end{pmatrix}\\
        &=\begin{pmatrix}
            e^{2\xi}(2x-y) + (-x+y)\cosh(\xi)+\tfrac{1}{-x+y}\sinh(\xi)\\
            e^{2\xi}(2x-y)+2(-x+y)\cosh(\xi)+\tfrac{2}{-x+y}\sinh(\xi)
        \end{pmatrix}.
    \end{align*}
\end{example}

\vspace{0.4cm}

\subsection{Solving systems of differential equations}
 As we have shown in the previous section, the tools we have developed can be useful in the field of differential equations. The following examples illustrate how to diagonalize nonlinear functions that model classical systems of nonlinear differential equations, and how to derive solutions to these problems.

Consider derivable functions $x(t),y(t)$ that satisfy the autonomous system
$$\begin{cases}
\frac{dx}{dt}=2e^{-x+y}-x+2y,\\
\frac{dy}{dt}=e^{-x+y}-x+2y.
\end{cases}$$
It can be written as $\frac{dX}{dt}=\Phi(X)$, being $X(t):=(x(t),y(t))$ and $\Phi$ the function defined in \eqref{eq:phiex}. As we have seen in Example \ref{ex:diagex}, this function can be diagonalized in the basis $\{(2,1),(1,1)\}$. In particular, its diagonal expression is $\Psi(p,q):=(e^{-p},q)$, with $p=x-y$ and $q=-x+2y$. Thus, the original system is transformed to
$$\begin{cases}
\frac{dp}{dt}=e^{-p}\\
\frac{dq}{dt}=q,
\end{cases}$$
that can be easily solved as two decoupled ODEs. For the second one, we obtain $q(t)=K_2e^t$. The first is also a separable equation with solution $p(t)=\log(t+K_1)$ when $t > -K_1$. So, the solutions for $x$ and $y$ are $x(t)=2\log(t+K_1)+K_2e^t$ and $y(t)=\log(t+K_1)+K_2e^t$. Indeed, note that
\begin{align*}
    \frac{dx}{dt}&=\frac{2}{t+K_1}+K_2e^t=2e^{-\log(t+K_1)}+K_2e^t=2e^{-x+y}-x+2y,\\
    \frac{dy}{dt}&=\frac{1}{t+K_1}+K_2e^t=e^{-\log(t+K_1)}+K_2e^t=e^{-x+y}-x+2y.
\end{align*}

On the other hand, in some contexts it is useful to invert the function $\Phi$, see \cite[Section 6.2.1]{soong1974}. In our example the diagonal function $\Psi$ can be easily inverted as $\Psi^{-1}(p,q)=(-\log(p),q)$, and so \begin{align*}
    \Phi^{-1}(x,y)&=(P\circ \Psi^{-1}\circ P^{-1})(x,y)\\
    &=(P\circ\Psi^{-1})(x-y,-x+2y)\\
    &=P\begin{pmatrix}
        -\log(x-y)\\ -x+2y
    \end{pmatrix}=\begin{pmatrix}
        -2\log(x-y)-x+2y\\ -\log(x-y)-x+2y
    \end{pmatrix},
\end{align*}
for $(x,y) \in \Phi(\mathbb R^2) = \{ (x,y) : x-y > 0 \}$.

\subsection{Study of dynamical systems}
To finish, let us show how diagonalization can be applied to study the equilibrium points of dynamical systems.

\begin{example}
    Consider the autonomous system
    \begin{equation}\label{eq:dynsystem}
        \left\{
        \begin{aligned}
            \frac{dx}{dt}&=2(2x-y)^{2}+\sin(-3x+2y),\\
            \frac{dy}{dt}&=3(2x-y)^{2}+2\sin(-3x+2y),
        \end{aligned}
        \right.
    \end{equation}
    which can be written as $dX/dt=\Phi(X)$, where $X(t):=(x(t),y(t))$ and
    $\Phi\colon\mathbb R^{2}\to\mathbb R^{2}$ is the $\mathcal C^{\infty}$ function
    \[
        \Phi(x,y)=\bigl(2(2x-y)^{2}+\sin(-3x+2y),\ 3(2x-y)^{2}+2\sin(-3x+2y)\bigr).
    \]

    Since $\mathbb R^{2}$ is open and convex, we can apply the procedure of
    Remark \ref{remark:method}. A direct computation gives
    \[
        q(x,y)=\rho(x,y)\,(-2,7,6), \qquad \text{with } \rho(x,y)=2(2x-y)-\cos(-3x+2y),
    \]
    which is not identically zero, since for instance $\rho(0,0)=-1$. As $(a,b,c)=(-2,7,6)$ satisfies $b^{2}+4ac=49-48=1>0$, Theorem \ref{teo:r2char} shows that $\Phi$ is diagonalizable. Since $a\neq0$, the equation $am^{2}+bm-c=-2m^{2}+7m-6=0$ has the two real solutions $m_{1}=2$ and $m_{2}=3/2$, so that $\{(1,2),(1,3/2)\}$ is a basis in which $\Phi$ is diagonal. By rescaling its second vector and reordering we can get the equivalent basis $\mathcal B=\{(2,3),(1,2)\}$.

    The corresponding change of basis matrices are
    \[
        P:=P_{\mathcal B,\mathcal C}=\begin{pmatrix}2&1\\3&2\end{pmatrix},
        \qquad
        P_{\mathcal C,\mathcal B}=P^{-1}=\begin{pmatrix}2&-1\\-3&2\end{pmatrix},
    \]
    hence the coordinates of $X$ with respect to $\mathcal B$ are $(u,v)=P^{-1}(x,y)=(2x-y,\,-3x+2y)$. As $\Phi(P(u,v))=\Phi(2u+v,\,3u+2v)=(2u^{2}+\sin v,\ 3u^{2}+2\sin v)$, the diagonal expression of $\Phi$ is
    \[
        \Psi(u,v)=(P^{-1}\circ\Phi\circ P)(u,v)
        =P^{-1}\begin{pmatrix}2u^{2}+\sin v\\[2pt] 3u^{2}+2\sin v\end{pmatrix}
        =\begin{pmatrix}u^{2}\\[2pt] \sin v\end{pmatrix}.
    \]

    Writing $U(t):=P^{-1}X(t)=(u(t),v(t))$ and using that $P^{-1}$ is linear, we get
    \[
        \frac{dU}{dt}=P^{-1}\frac{dX}{dt}=P^{-1}\Phi(X)=P^{-1}\Phi(PU)=\Psi(U),
    \]
    so that in the new coordinates the system \eqref{eq:dynsystem} becomes the decoupled pair of scalar equations
    \begin{equation}\label{eq:dyndiag}
        \frac{du}{dt}=u^{2},\qquad \frac{dv}{dt}=\sin v .
    \end{equation}

    Since $P$ is a linear isomorphism, $X_{0}$ is an equilibrium point of \eqref{eq:dynsystem} if and only if $P^{-1}X_{0}$ is an equilibrium point of \eqref{eq:dyndiag}. The equilibria of \eqref{eq:dyndiag} are the points $(u,v)=(0,k\pi)$ with $k\in\mathbb Z$, and consequently the equilibrium points of the original system are
    \[
        (x,y)=P(0,k\pi)=(k\pi,\,2k\pi),\qquad k\in\mathbb Z .
    \]
\end{example}

\section*{Acknowledgements}
The first author thanks the support of Generalitat Valenciana (Spain), grant number PROMETEO 2024 CIPROM/2023/32.
The third author was supported by a contract of the Programa de Ayudas de Investigación y Desarrollo (PAID-01-24), Universitat Politècnica de València. 
This work has been partially supported by the grant PID2025-172735NB-I00 funded by MICIU/AEI/10.13039/501100011033, by the grant CIPROM/2026/22
funded by the Generalitat Valenciana (Conselleria de Educación, Cultura, Universidades y Empleo), and by the grant PID2022-138342NB-I00 funded by MCIN/ AEI/10.13039/501100011033.

\section*{Declarations}
The authors have no competing interests to declare that are relevant to the content of this article.

\printbibliography

\end{document}